\documentclass[11pt]{article}

\usepackage{amssymb,amsmath}
\usepackage{amsthm}
\usepackage{hyperref}
\usepackage{mathrsfs}
\usepackage{textcomp}

\newtheorem{thm}{Theorem}[section]
\newtheorem{lem}{Lemma}[section]
\newtheorem{prop}{Proposition}[section]
\newtheorem{cor}{Corollary}[section]

\theoremstyle{definition}

\newtheorem{rem}[thm]{Remark}
\newcommand{\al}{\alpha}                \newcommand{\lda}{\lambda}
\newcommand{\om}{\Omega}                \newcommand{\pa}{\partial}
\newcommand{\va}{\varepsilon}           \newcommand{\ud}{\mathrm{d}}
\newcommand{\be}{\begin{equation}}      \newcommand{\ee}{\end{equation}}
                 \newcommand{\Sn}{\mathbb{S}^n}

\newcommand{\R}{\mathbb{R}}
\newcommand{\abs}[1]{\lvert#1\rvert}

\author{Tianling Jin and Jingang Xiong}
\title{\textbf{A fractional Yamabe flow and some applications}}
\begin{document}
% generates the title

%\date{}

\maketitle

\begin{abstract}
We introduce a fractional Yamabe flow involving nonlocal conformally invariant operators on the conformal infinity of asymptotically hyperbolic manifolds, and show that on the conformal spheres $(\Sn, [g_{\Sn}])$, it converges to the standard sphere up to a M\"obius diffeomorphism. These arguments can be applied to obtain
extinction profiles of solutions of some fractional porous medium equations. 
In the end, we use this fractional fast diffusion equation, together with its extinction profile and some estimates of its extinction time, to improve a 
 Sobolev inequality via a quantitative estimate of the remainder term.
\end{abstract}

%\tableofcontents

% 35R11 (35B40, 35K57, 46E35)

\section{Introduction}

Let $(M,g_0)$ be a compact Riemannian manifold of dimension $n\geq 2$. The following evolution equation for the metric $g$
\be\label{classic yamabe flow}
\frac{\pa}{\pa t}g(t)=-(R_{g(t)}-r_{g(t)})g(t),\quad g(0)=g_0
\ee
was introduced by Hamilton in \cite{Ham}, and is known as the Yamabe flow. Here, $R_{g(t)}$ is the scalar curvature of $g(t)$ and $r_{g(t)}=vol_{g(t)}(M)^{-1}\int_M R_{g(t)} \ud vol_{g(t)}$ is the average of $R_{g(t)}$. The existence and convergence of solutions of \eqref{classic yamabe flow} were established through \cite{Ham}, \cite{Chow92}, \cite{Y}, \cite{SS}, \cite{B4} and \cite{B07a}.

In \cite{GJMS}, Graham, Jenne, Mason and Sparling constructed a sequence of conformally invariant elliptic operators, $\{P_{k}^g\}$, on $(M,g)$ for all positive integer $k$ if $n$ is odd, and for $k\in \{1,\cdots, n/2\}$ if $n$ is even.
Moreover, $P_1^g$ is the conformal Laplacian $L_g:=-\Delta_{g}+c(n)R_g$ and $P_2^g$ is the Paneitz operator. The construction in \cite{GJMS} is based on the ambient metric construction of \cite{FG}. Up to a positive constant
$P_1^g(1)$ is the scalar curvature of $g$ and $P_2^g(1)$ is the $Q$-curvature. 
Some higher integer order curvature flows involving $P^g_k$, such as 
$Q$-curvature flow, 
have been studied in \cite{B03, MS, BFR, B06, Ho} and so on.

Making use of a generalized Dirichlet to Neumann map, Graham and Zworski \cite{GZ} introduced a meromorphic family of conformally invariant operators on the conformal infinity of asymptotically hyperbolic manifolds (see Mazzeo and Melrose \cite{MM87}). Recently,  Chang and Gonz\'alez \cite{CG} reconciled the way of Graham and Zworski to define conformally invariant operators $P_{\sigma}^g$ of non-integer order $\sigma\in (0,\frac{n}{2})$ and the localization method of  Caffarelli and Silvestre \cite{CS} for factional Laplacian $(-\Delta)^{\sigma}$ on the Euclidean space $\R^n$. These lead naturally to a fractional order curvature $R^g_{\sigma}:=P_{\sigma}^g(1)$, which will be called $\sigma$-curvature in this paper.  There have been several work on these conformally invariant equations of fractional order and prescribing $\sigma$-curvatures problem (fractional Yamabe problem and fractional Nirenberg problem), see, e.g., \cite{QR}, \cite{GMS}, \cite{GQ}, \cite{JLX1} and \cite{JLX2}. In this paper we study some flow of this fractional order curvature $R_{\sigma}^g$ associated with $P_{\sigma}^g$ on the standard conformal sphere $(\Sn,[g_{\Sn}])$, which is the conformal infinity of the Poincar\'e disk with the standard Poincar\'e metric. 

Let $g$ be a representative in the conformal class $[g_{\Sn}]$ and write $g=v^{\frac{4}{n-2\sigma}}g_{\Sn}$, where $v$ is positive and smooth on $\Sn$. 
Then we have
\be\label{def of sig cur}
P_{\sigma}^{g}(\phi)=v^{-\frac{n+2\sigma}{n-2\sigma}}P_{\sigma}^{g_{\Sn}}(\phi v)\quad \mbox{for any }\phi\in C^{\infty}(\Sn).
\ee
$P_{\sigma}^{g_{\Sn}}$, which is simply written as $P_{\sigma}$, has the formula
(see, e.g., \cite{Br})
\be\label{P sigma}
 P_\sigma=\frac{\Gamma(B+\frac{1}{2}+\sigma)}{\Gamma(B+\frac{1}{2}-\sigma)},\quad B=\sqrt{-\Delta_{g_{\Sn}}+\left(\frac{n-1}{2}\right)^2},
\ee
where $\Gamma$ is the Gamma function and $\Delta_{g_{\Sn}}$ is the Laplace-Beltrami  operator on $(\Sn, g_{\Sn})$. 
Let $Y^{(k)}$ be a spherical harmonic of degree $k\ge 0$. 
Since $-\Delta_{g_{\Sn}}Y^{(k)}=k(k+n-1) Y^{(k)}$, 
\be\label{value on spherical harmonics}
 B\Big(Y^{(k)}\Big)=\left(k+\frac{n-1}{2}\right)Y^{(k)}\quad\mbox{and}\quad P_\sigma\Big(Y^{(k)}\Big)=\frac{\Gamma(k+\frac{n}{2}+\sigma)}{\Gamma(k+\frac{n}{2}-\sigma)}Y^{(k)}. 
\ee

It is also well-known (see, e.g., \cite{Mo}) that $P_{\sigma}$ is the inverse of the spherical Riesz potential
\be\label{P sigma inverse}
K^{\sigma}(f)(\xi)=c_{n,\sigma}\int_{\Sn}\frac{f(\zeta)}{\abs{\xi-\zeta}^{n-2\sigma}}\,\ud vol_{g_{\Sn}}(\zeta),\quad f\in L^p(\Sn)
\ee
where $c_{n,\sigma}=\frac{\Gamma(\frac{n-2\sigma}{2})}{2^{2\sigma}\pi^{n/2}\Gamma (\sigma)}$, $1\leq p<\infty$ and $|\cdot|$ is the Euclidean distance in $\R^{n+1}$.
The inverses of spherical Riesz potentials have been constructed in terms of singular integrals in \cite{PS} and \cite{R}. 
When $\sigma\in (0,1)$, Pavlov and Samko \cite{PS} showed that if $v=K^{\sigma}(f)$ for some $f\in L^p(\Sn)$, then
\be\label{description of P sigma}
P_{\sigma}(v)(\xi)=P_{\sigma}(1)v(\xi)+c_{n,-\sigma}\int_{\Sn}\frac{v(\xi)-v(\zeta)}{\abs{\xi-\zeta}^{n+2\sigma}}\,\ud vol_{g_{\Sn}}(\zeta),
\ee
 where $c_{n,-\sigma}=\frac{2^{2\sigma}\sigma\Gamma(\frac{n+2\sigma}{2})}{\pi^{\frac{n}{2}}\Gamma(1-\sigma)}$ and $\int_{\Sn}$ is understood as $\lim\limits_{\va\to 0}\int_{|x-y|>\va}$ in $L^p(\Sn)$ sense. 
 
 Consider the normalized total $\sigma$-curvature functional
\[
S(g)=vol_g(\Sn)^{\frac{2\sigma-n}{n}}\int_{\Sn} R_{\sigma}^g\,\ud vol_g,\quad g\in [g_{\Sn}].
\]
The negative gradient flow of $S$ takes the form
\[
\frac{\pa g}{\pa t}=-\frac{n-2\sigma}{2n}( vol_g(\Sn))^{\frac{2\sigma-n}{n}}(R_{\sigma}^g-r_{\sigma}^g)g,
\]
where $r_{\sigma}^g$ is the average of $R_{\sigma}^g$. It is easy to verify that this flow preserves the conformal class and the volume of $\Sn$. By a rescaling of the time variable, we obtain the following evolution equation
\be\label{Y1}
\frac{\pa g}{\pa t}=-(R_{\sigma}^g-r_{\sigma}^g)g.
\ee
If we write $g(t)=v^{\frac{4}{n-2\sigma}}(\cdot,t)g_{\Sn}$,
then after rescaling the time variable, \eqref{Y1} can be written in an equivalent form
\be\label{Y2}
\frac{\pa v^{N}}{\pa t}=
-P_\sigma(v)+r_{\sigma}^g v^{N} \quad \mbox{on }\Sn,
\ee
where $N=(n+2\sigma)/(n-2\sigma)$. 

Let $\mathcal N$ be the north pole of $\Sn$ and 
\[
F: \R^n\to \Sn\setminus\{\mathcal N\}, \quad x\mapsto \left(\frac{2x}{1+|x|^2}, \frac{|x|^2-1}{|x|^2+1}\right)
\]
be the inverse of stereographic projection from $\Sn\backslash \{\mathcal N\}$ to $\R^n$. Then
\be\label{P-sigma to frac lap}
(P_\sigma(\phi))\circ F=  |J_F|^{-\frac{n+2\sigma}{2n}}(-\Delta)^\sigma(|J_F|^{\frac{n-2\sigma}{2n}}(\phi\circ F))\quad \mbox{for }\phi\in C^{\infty}(\Sn),
\ee
where
$
|J_F|=\left(\frac{2}{1+|x|^2}\right)^n
$
and $(-\Delta)^\sigma$ is the fractional Laplacian operator (see, e.g., \cite{Stein}).
Hence $u(x,t):=|J_F|^{\frac{n-2\sigma}{2n}}v(F(x),t)$ satisfies
\be\label{Y3}
\frac{\pa u^{N}}{\pa t}=
-(-\Delta)^\sigma u + r_{\sigma}^g u^{N} \quad \mbox{in }\R^n.
\ee
We will call \eqref{Y1}, \eqref{Y2} or \eqref{Y3} as a (normalized)
\emph{fractional Yamabe flow} when $\sigma\in (0,1)$. 

As observed in \cite{CG} that the operator $P_{1/2}^g$ is related to the Yamabe problem on manifolds with boundary (see, e.g., \cite{Ch, E1, E2, HL1, HL2}), this fractional Yamabe flow \eqref{Y1} with $\sigma=1/2$ is related to some generalization of Yamabe flow for manifolds with boundary studied in \cite{B02}.

 Throughout this paper we always assume that $0<\sigma<1$ without otherwise stated.
Our first result is the long time existence and convergence of solutions of \eqref{Y1} for any initial data in the conformal class of $g_{\Sn}$.

\begin{thm}\label{convergence of fractional yamabe flow on spheres}
Let $g(0)\in [g_{\Sn}]$ be a smooth metric on $\Sn$ for $n\ge 2$. Then the fractional Yamabe 
flow \eqref{Y1} with initial metric $g(0)$ exists for all time $0<t<\infty$. Furthermore, there exists a smooth metric $g_{\infty}\in [g_{\Sn}]$ such that
\[
R^{g_{\infty}}_{\sigma}=r^{g_{\infty}}_{\sigma}\quad\text{and}\quad\lim_{t\to\infty}\|g(t)-g_{\infty}\|_{C^l(\Sn)}=0
\]
for all positive integers $l$. 
\end{thm}

\begin{rem} 
If we write $g_{\infty}=v_{\infty}^{\frac{4}{n-2\sigma}}g_{\Sn}$ where $v_{\infty}$ is a smooth and positive function on $\Sn$,  then Theorem \ref{convergence of fractional yamabe flow on spheres} implies that $v_{\infty}$ satisfies
\[
P_{\sigma}(v_{\infty})=r^{g_{\infty}}_{\sigma}\cdot v_{\infty}^{\frac{n+2\sigma}{n-2\sigma}},
\]
whose solutions are classified in \cite{CLO} and \cite{Li04}.
\end{rem}

We also consider the unnormalized fractional Yamabe flow
\[
\frac{\pa v^{N}}{\pa t}=-P_\sigma(v)\quad \mbox{on }\Sn\times (0,\infty),
\quad\text{or}\quad
\frac{\pa u^{N}}{\pa t}=-(-\Delta)^\sigma u\quad \mbox{in }\R^n\times (0,\infty).
\]
The second one is a fractional porous medium equation studied, e.g., in \cite{AC, PV, CV, dQRV, KL}, where it is taken the form
\begin{equation}\label{fpme}
\begin{cases}
\begin{aligned}
       u_t&=-(-\Delta)^{\sigma}(|u|^{m-1}u)&\quad &\text{in } \R^n\times(0,\infty),\\
       u(\cdot,0)&=u_0&\quad &\text{in }\R^n,
       \end{aligned}
\end{cases}
\end{equation}
with $m=\frac{n-2\sigma}{n+2\sigma}$, $\sigma\in (0,1)$. Models of this kind of fractional diffusion equations arise, e.g., in statistical mechanics \cite{Jara, Jara2, JKO} and heat control \cite{AC}.

These fractional diffusion equations have been systematically studied in \cite{PV} and \cite{dQRV}. It is proved in \cite{dQRV} that if $u_0\in L^1(\R^n)\cap L^p(\R^n)$ for $p>4n/(n+2\sigma)$, then there exists a unique strong solution (see \cite{dQRV} for the definition) of \eqref{fpme}, and the solution will extinct in finite time.
More precisely, if $u_0$ is nonnegative but not identically equals to zero, then there exists a $T=T(u_0)\in (0,\infty)$ such that $u(x,t)>0$ in $\R^n\times (0,T)$, and $u(x,T)\equiv 0$ in $\R^n$. 

We are interested in analyzing the exact behavior of solutions of \eqref{fpme} near the extinction time for fast decaying initial data. 
In the classical case $\sigma=1$, the extinction profiles of solutions of porous medium equations have been described in the results of \cite{GP, ds, DS, BBDGV, BDGV} and so on.

Theorem \ref{extinction profile} below describes the extinction profile of $u(x,t)$, which extends the result of del Pino and Sa\'ez in \cite{ds} for $\sigma=1$ to $\sigma\in (0,1)$. 

\begin{thm}\label{extinction profile}
Assume that $u_0\in C^{2}(\R^n)$ is positive in $\R^n$ for $n\ge 2$. In addition, we assume, for $(u_0^m)_{0,1}(x):=\abs x^{2\sigma-n}u_0^m(x/\abs x^2)$,
 that $(u_0^m)_{0,1}(x)$ can be extended to a positive and $C^2$ function near the origin. There exist $\lda>0$ and $x_0\in R^n$ such that if $T=T(u_0)\in (0,\infty)$ denotes the extinction time of the solution of \eqref{fpme}, then
\[
(T-t)^{-1/(1-m)}u(x,t)=k(n,\sigma)\left(\frac{\lda}{\lda^2+\abs{x-x_0}^2}\right)^{\frac{n+2\sigma}{2}} + \theta(x,t)
\]
with
\[
\sup_{\R^n}(1+\abs x^{n+2\sigma})\theta(x,t)\to 0\quad \text{as }t\to T,
\]
where $k(n,\sigma)=2^{\frac{n-2}{2}}\big((1-m)P_{\sigma}(1)\big)^{\frac{n-2\sigma}{4\sigma}}$ and $P_{\sigma}(1)$ is given in \eqref{value on spherical harmonics}.
\end{thm}

Some estimates of the extinction time $T$ involving the sharp constant in Sobolev inequalities are postponed to Lemma \ref{estimate of extinction time} in Section \ref{inequality along flow}.

An application of Theorem \ref{extinction profile} is an improvement of some Sobolev inequality. A sharp form of the standard Sobolev inequality in $\R^n$ $(n\ge 3)$ asserts that
\be\label{sse}
S_{n}\|\nabla u\|_{L^2(\R^n)}-\|u\|_{L^{\frac{2n}{n-2}}(\R^n)}\geq 0
\ee
for all $u\in \dot H^{1}(\R^n)=\{u\in L^{\frac{2n}{n-2}}(\R^n): \nabla u \in L^2(\R^n)\}$, where $S_{n}$ is the sharp constant obtained in \cite{aubin} and \cite{T}. 

There have been many results on remainder terms of Sobolev inequalities  (see, e.g., \cite{BN, BL, BE, CFMP, CCL, D}), which give various lower bounds of the left-hand side of \eqref{sse}.

Recently, Carlen, Carrillo and Loss in \cite{CCL} noticed that some Hardy-Littlewood-Sobolev inequalities in dimension $n\geq 3$ and some special Gagliardo-Nirenberg inequalities can be related by a fast diffusion equation. In another recent paper \cite{D}, Dolbeault used a fast diffusion flow to obtain an optimal integral remainder term which improves \eqref{sse} in dimension $n\geq 5$. Inspired by \cite{CCL} and \cite{D}, we consider some Sobolev inequality involving fractional Sobolev spaces of order $\sigma\in(0,1)$, compared to those mentioned above corresponding to $\sigma=1$.

For any $\sigma\in(0,1)$, the Sobolev inequality (see, e.g., \cite{Stein} or \cite{DPV}) asserts that
\be\label{eq:fs}
\|u\|^2_{L^{2^*(\sigma)}}\leq S_{n,\sigma}\|u\|^2_{\dot H^{\sigma}},\quad \forall \ u \in \dot H^{\sigma}(\R^n),
\ee
where $2^*(\sigma)=\frac{2n}{n-2\sigma}$, $S_{n,\sigma}$ is the optimal constant and $\dot H^{\sigma}(\R^n)$ is the closure of $C_c^{\infty}(\R^n)$ under the norm 
\be\label{hom frac norm}
\|u\|_{\dot H^{\sigma}}=\|(-\Delta)^{\sigma/2} u\|_{L^2(\R^n)}.
\ee 
The optimal constant $S_{n,\sigma}$ in the Sobolev inequality \eqref{eq:fs} is obtained by Lieb \cite{Lie83} and is achieved by
$u(x)=\left(1+\abs x^2\right)^{-\frac{n-2\sigma}{2}}.$
The Hardy-Littlewood-Sobolev inequality
\be\label{eq:HLS}
S_{n,\sigma}\|u\|^2_{L^{\frac{2n}{n+2\sigma}}}\geq\int_{\R^n}u(-\Delta)^{-\sigma}u\,\ud x,\quad \forall \ u \in L^{\frac{2n}{n+2\sigma}}(\R^n)
\ee
involves the same optimal constant $S_{n,\sigma}$, where $(-\Delta)^{-\sigma}$ is a Riesz potential defined by
\be\label{sigma potential}
(-\Delta)^{-\sigma}u(x)=c_{n,\sigma}\int_{\R^n}\frac{u(y)}{\abs{x-y}^{n-2\sigma}}\ud y.
\ee

We investigate the relation between \eqref{eq:fs} and \eqref{eq:HLS} via the fractional diffusion equation \eqref{fpme}, i.e.
\begin{equation*}
       u_t=-(-\Delta)^{\sigma}u^{m}
\end{equation*}
 with $m=1/N=\frac{n-2\sigma}{n+2\sigma}$. If we suppose that the initial data $u_0$ satisfies the assumptions in Theorem \ref{extinction profile}, then by Theorem \ref{convergence of v} (which is used to prove Theorem \ref{extinction profile}) $u(\cdot,t)$ is positive and smooth in $\R^n$ before its extinction time, and for any fixed $t$, $u(x,t)=O(\abs x^{-n-2\sigma})$ as $x\to\infty$. We define
 \be\label{definition of H}
 H(t):=H_{n,\sigma}(u(\cdot,t))=\int_{\R^n}u(-\Delta)^{-\sigma}u\ud x-S_{n,\sigma}\|u\|^2_{L^{\frac{2n}{n+2\sigma}}}.
 \ee
It follows from direct computations that $\frac{\ud}{\ud t} H\geq 0$ (see Proposition \ref{increase H}).

Consequently, one can prove \eqref{eq:HLS}, which is equivalent to $H\leq 0$, by showing
\[
\limsup_{t\to T} H(t)\leq 0,
\] where $T$ is the extinction time of \eqref{fpme}. This can be seen clearly from Theorem \ref{convergence of v}. From this and Proposition \ref{increase H} we also recover that $u^m$ is an extremal of \eqref{eq:fs} if $u$ is an extremal of \eqref{eq:HLS}. 

Along this fractional fast diffusion flow, we can improve the Sobolev inequality \eqref{eq:fs}, via a quantitative estimate of the remainder term. This improvement also holds as $\sigma\to 1$ and it extends some work of Dolbeault in \cite{D}.

\begin{thm}\label{duality}
Assume that $\sigma\in (0,1)$ and $n>4\sigma$. There exists a positive constant $C$ depending only on $n$ and $\sigma$ such that
for any nonnegative function $u\in\dot H^{\sigma}(\R^n)$ we have
\be\label{eq:estimate for the remainder term}
\begin{split}
S_{n,\sigma}\|u^N\|^2_{L^{\frac{2n}{n+2\sigma}}}&-\int_{\R^n}u^N(-\Delta)^{-\sigma}u^N\ud x \\
&\leq C\|u\|_{L^{2^*(\sigma)}}^{\frac{8\sigma}{n-2\sigma}}\left(S_{n,\sigma}\|u\|^2_{\dot H^{\sigma}}-\|u\|^2_{L^{2^*(\sigma)}}\right),
\end{split}
\ee
where $N=\frac{n+2\sigma}{n-2\sigma}$. Moreover, $C$ can be taken as $\frac{n+2\sigma}{n}(1-e^{-\frac{n}{2\sigma}})S_{n,\sigma}$.
\end{thm}

The operators $P_{\sigma}$ and $(-\Delta)^{\sigma}$ are nonlocal, pseudo-differential operators. Generally speaking, strong maximum principles and Harnack inequalities might fail for nonlocal operators, see, e.g., a counterexample in \cite{Ka2}. The counterexample in \cite{Ka2} shows that the local non-negativity of solutions of certain nonlocal equations is not enough to guarantee local strong maximum principles and Harnack inequalities. However, if solutions are assumed to be globally nonnegative, then various strong maximum principles and Harnack inequalities have been obtained in, e.g., \cite{CaS}, \cite{TX} and \cite{JLX1}.

In this paper, we establish a strong maximum principle and a Hopf lemma for odd solutions of some linear nonlocal parabolic equations, which should be of independent interest. Our proofs make use of the expression \eqref{fractional laplacian expression in integral} of $(-\Delta)^{\sigma}$. The odd function in Lemma \ref{lemtf} will serve as a barrier function, which allows us to obtain a Hopf lemma.

The paper is organized as follows. In Section \ref{max and hopf} we prove a strong maximum principle and a Hopf lemma for odd solutions of some linear nonlocal parabolic equations. These are two crucial ingredients in our arguments. In Section \ref{harnack} we prove a Harnack inequality via the method of moving planes. This idea is essentially due to R. Ye \cite{Y}. In Section \ref{section: existence and convergence}, we show Schauder estimates, existence and convergence of solutions of the fractional Yamabe flows. Section \ref{Extinction profile of a fractional porous medium equation} is devoted to proving Theorem \ref{extinction profile}. We rewrite \eqref{fpme} on $\Sn$ and apply the methods in Section \ref{section: existence and convergence}. The improvement of the Sobolev inequality, Theorem \ref{duality}, is proved in Section \ref{inequality along flow}. Our proofs of Theorem \ref{extinction profile} and Theorem \ref{duality} adapt some arguments in \cite{ds} and \cite{D}, respectively. In Appendix \ref{appendix a}, we provide an analog of L. Simon's uniqueness theorem (see \cite{S}) for negative gradient flows in our nonlocal flow setting. In Appendix \ref{appendix b} we present some interpolation inequalities and elementary computations which are used in Section \ref{yambe}.

\bigskip

\noindent\textbf{Acknowledgements:} Both authors thank Prof. Y.Y. Li for encouragements and useful discussions. We also thank Prof. N. Sesum for bringing paper \cite{ds} to our attention. Tianling Jin was partially supported by a University and Louis Bevier Dissertation Fellowship at Rutgers University and Rutgers University School of Art and Science Excellence Fellowship. Jingang Xiong was partially supported by
CSC project for visiting Rutgers University and NSFC No. 11071020. He is very grateful to the Department of Mathematics at Rutgers University for the kind hospitality.

\bigskip

\textbf{Added in April 2012:} After the first author gave a seminar at Princeton University in April 2012, Professor Rupert L. Frank informed us that they studied strong maximum principles and Hopf lemma for odd solutions for nonlocal elliptic equations in \cite{FL}. This also leads us to another reference \cite{BLW} where some extension of Hopf lemma to non-local contexts is proved. Our case and the proofs are different from theirs. We thank Professor Rupert L. Frank for his interests in our work and bringing the references \cite{BLW, FL} to our attention.

\section{A strong maximum principle and a Hopf lemma for nonlocal parabolic equations}\label{max and hopf}

Let $x=(x',x_n)\in \R^n$, $\R^n_+=\{x:x_n>0,x\in \R^n\}$. 
Recall (see, e.g., \cite{silvestre}) that for $\sigma\in (0,1)$, if $u$ is bounded in $\R^n$ and $C^{2}$ near $x$, then$ (-\Delta)^{\sigma}u$ is continuous near $x$, and
\be\label{fractional laplacian expression in integral}
 (-\Delta)^{\sigma}u(x)=c_{n,-\sigma}\mbox{P.V.}\int_{\R^n}\frac{u(x)-u(y)}{\abs{x-y}^{n+2\sigma}}\ud y.
\ee
Here ``P.V." means the principal value and $c_{n,-\sigma}$ is the constant in \eqref{description of P sigma}.

For simplicity, throughout the paper we denote $-(-\Delta)^\sigma$ as $\Delta^\sigma$ and will not keep writing the constant $c_{n,-\sigma}$ and ``P.V.'' if there is no confusion.

\begin{lem}\label{wmp}
Let $w(x,t)\in C^{2,1}(\R^n\times\R)$ and $w(\cdot,t)$ be bounded in $\R^n$ for any fixed $t$. Suppose $w(x,t)$ satisfies 
$w(x',-x_n,t)=-w(x',x_n,t)$ for all $(x,t)$ and
\[
\liminf_{x_n\geq 0, |x|\to \infty}w(x,t)\geq 0\quad\mbox{for any fixed }t.
\]
If $w$ satisfies
 \be\label{eq: linear}
w_{t}\geq a(x,t)\Delta^{\sigma} w + b(x,t) w, \quad (x,t)\in \R^n_+\times (0, T],
 \ee
 where $a(x,t)$ is continuous and positive in $\overline{\R^n_+}\times [0,T]$, $b(x,t)$ is continuous and bounded in $\R^n_+\times [0,T]$, 
and $w(x,0)\geq 0$ for all $x\in \R^n_+$, then $w(x,t)\geq0$ in $\R^n_+\times [0,T]$.
 \end{lem}

 \begin{proof} Without loss of generality, we may assume $b(x,t)\leq 0$. Indeed,
if we let
 \[\tilde w(x,t)=e^{-Ct}w(x,t)
 \]
for some constant $C$,  then
 \[
 \tilde w_{t}=a(x,t)\Delta^{\sigma} \tilde w + (b(x,t)-C) \tilde w.
 \]
 Since $b$ is bounded, we can choose $C$ sufficiently large such that $b(x,t)-C\leq 0$ in $\R^n_+\times (0,T]$, and we
 only need to show that $\tilde w(x,t)\geq 0$ for all $(x,t)\in \R^n_+\times (0,T]$.

Suppose the contrary that there exists a point $(x_0,t_0)\in \R^n_+\times (0,T]$ such that
\[
0>  w(x_0,t_0).
\]
By the assumptions on $w$, we may assume $w(x_0,t_0)=\min_{\overline {\R^n_+}\times (0,T]}  w$. 
It follows that
\be\label{mon1}
w_t(x_0,t_0)\leq 0, \quad  b(x_0,t_0) w(x_0,t_0)\geq 0.
\ee
It is clear that
\[
\begin{split}
 \Delta^\sigma  w(x_0,t_0)
&=\int_{\R^n}\frac{ w(y,t_0)- w(x_0,t_0)}{|x_0-y|^{n+2\sigma}}\,\ud y\\
&=\int_{\R^n_+}\frac{ w(y,t_0)- w(x_0,t_0)}{|x_0-y|^{n+2\sigma}}\,\ud y+\int_{\R^n\setminus \R^n_+}
\frac{ w(y,t_0)- w(x_0,t_0)}{|x_0-y|^{n+2\sigma}}\,\ud y.
\end{split}
\]
By the change of variables $y=(z',-z_n)$, we obtain
\[
 \begin{split}
  &\int_{\R^n\setminus \overline{\R^n_+}}\frac{ w(y,t_0)- w(x_0,t_0)}{|x_0-y|^{n+2\sigma}}\,\ud y\\&
 =\int_{ \R^n_+}\frac{ w(z',-z_n,t_0)- w(x_0,t_0)}{|x_0-(z',-z_n)|^{n+2\sigma}}\,\ud z\\
&=-\int_{ \R^n_+}\frac{ w(z',z_n,t_0)- w(x_0,t_0)}{|x_0-(z',-z_n)|^{n+2\sigma}}\,\ud z -2w(x_0,t_0)\int_{ \R^n_+}
\frac{1}{|x_0-(z',-z_n)|^{n+2\sigma}}\,\ud z\\
&> -\int_{ \R^n_+}\frac{ w(z',z_n,t_0)- w(x_0,t_0)}{|x_0-(z',-z_n)|^{n+2\sigma}}\,\ud z.
 \end{split}
\]
where we used $ w(z',-z_n,t_0)=- w(z',z_n,t_0)$ and $ w(x_0,t_0)< 0$.
Since $(x_0, t_0)$ is a minimum point of $w$ in $\overline {\R^n_+}\times (0,T]$, the simple inequality
\[
 \frac{1}{|x_0-z|^{n+2\sigma}}>\frac{1}{|x_0-(z',-z_n)|^{n+2\sigma}},\quad  \forall\ x_0,z\in \R^n_+
\]
yields that
\be\label{mon2}
 \Delta^\sigma w(x_0,t_0)>0.
\ee
Combining \eqref{mon1} and \eqref{mon2}, we have a contradiction to \eqref{eq: linear}.
\end{proof}

\begin{lem}\label{w-s mp}
Let $w(x,t)$ be as in Lemma \ref{wmp}. Then for any fixed $t\in (0,T]$, we have $w(x,t)>0$ or $w(x,t)\equiv0$ in $\R^n_+$.
\end{lem}

\begin{proof}
As in the proof of Lemma \ref{wmp}, we may assume $b\leq 0$.
Suppose that at $w(x_0,t_0)=0$ for some $(x_0,t_0)\in \R^n_+\times (0,T] $. From the proof of Lemma \ref{wmp} we see that
\[
\Delta^{\sigma} w(x_0,t_0) \geq 0
\]
and equality holds if and only if $ w(x,t_0)= w(x_0,t_0)$ for all $x\in \R^n_+$. Therefore,
the lemma follows immediately from a simple contradiction argument.
\end{proof}

\begin{lem}\label{smp}
Let $w(x,t)$ be as in Lemma \ref{wmp}. Suppose $w(x_0,0)>0$ for some $x_0\in \R^n_+$, then for any fixed $t\in (0,T]$, we have $w(x,t)>0$ in $\R^n_+$.
\end{lem}

\begin{proof}
The proof follows from that of the parabolic strong maximum principle in \cite{N}, with suitable modifications for nonlocal equations. As before, we assume $b\leq 0$. Suppose that for some $t_1>0$, $w(\cdot, t_1)$ is zero at some point. It follows from Lemma \ref{w-s mp} that $ w(x, t_1)\equiv 0$. By the assumption on $w(\cdot, 0)$ and Lemma \ref{w-s mp}, we may assume that $w(x,t)>0$ for all $(x,t)\in \R^n_+\times (t_2,t_1)$ for some $t_2>0$.

Let $h(x,t)=(t_1-t_*)^2-|x-e_n|^2-(t-t_*)^2$ if $0\leq x_n\leq 2$; and $h=0$ if $x_n> 2$, where $e_n=(0',1)$ and $t_*$ will be fixed later.
Set
\[
H(x,t)=\left\{
             \begin{array}{ll}
                h(x,t), & x\in \overline{\R^n_+}, \\
                -h(x',-x_n,t), & x\in\R^n\setminus \overline{\R^n_+}. \\
             \end{array}
           \right.
\]
Let $\bar t\in (t_2, t_1)$ be such that $(t_1-t_*)^2-(t-t_*)^2\leq \frac{1}{4}$ holds for $t\geq \bar t$.
It is easy to see that there exists a positive constant $C_1$ independent of $t^*$ such that for any $(x,t)\in B_{1/2}(e_n)\times [\bar t, t_1],$
\[
(-\Delta)^{\sigma}H(x,t)\leq C_1.
\]
Thus, we can choose $t_*$ so negative that for any $(x,t)\in B_{1/2}(e_n)\times [\bar t, t_1],$
\be\label{h}
H_t(x,t)=2(t_*-t)\leq 2(t_*-t_2)< a(x,t)\Delta^{\sigma}H(x,t)+b(x,t)H(x,t).
\ee
Let $\va>0$ be a sufficiently small constant such that $ w(x,\bar t)\geq \va H(x,\bar t)$ for all $x\in \R^n_+$.
We claim that $w(x,t)\geq\va H(x,t)$ in $\R^n_+\times (\bar t, t_1)$.

If not, then the (negative) minimal value of $\bar w:=w-\va H$ in $\R^n_+\times (\bar t,t_1)$ must be achieved in
$B_{1/2}(e_n)\times (\bar t, t_1)$, say at $(x_0,t_0)$. Note that $\bar w(x',-x_n,t)=-\bar w(x',x_n,t)$. Hence, by exactly the same argument in the proof of Lemma \ref{wmp}
\[
\pa_t  \bar  w(x_0,t_0)\leq 0,\quad \Delta^\sigma\bar w(x_0,t_0)>0.
\]
Together with \eqref{h} and $b\leq 0$, we conclude that
\[
 w_t(x_0,t_0)<a(x_0,t_0)\Delta^\sigma w(x_0,t_0)+b(x_0,t_0)w,
\]which contradicts \eqref{eq: linear}.

Hence, it follows from the above claim that $w_t(t_1, e_n)\leq -2\va(t_1-t_*)<0$. But $w(x,t_1)\equiv 0$. These contradict \eqref{eq: linear}.
\end{proof}

\begin{lem}\label{lemtf}
 Let
\[
 h(x)=\begin{cases}
       x_n(1-|x'|^2), &\quad |x_n|<1,|x'|<1,\\
      0,&\quad \mbox{otherwise}.
      \end{cases}
\]
Then there exists a positive constant $c_0$ depending only $n,\sigma$ such that
\be\label{ITF}
 \Delta^\sigma h(x)\geq -c_0 x_n,
\ee
for all $x=(x',x_n)$ with $|x'|<1, 0\leq x_n<1/8$.
\end{lem}

\begin{proof}
The lemma follows from rather involved calculations. By rotating the first $(n-1)$ coordinates, we only need to show \eqref{ITF} at point
$a=(a_1,0,\cdots, 0,a_n)$ with $0\leq a_1<1$, $0\leq a_n<1/8$.

Denote $B'(x',R)\subset \R^{n-1}$ be the ball centered at $x'$ with radius $R$, $\om=B'(0,1)\times (-1,1)$. In the following $C$ will be denoted as various positive constants which 
depend only on dimension $n$ and $\sigma$.

It follows from \eqref{fractional laplacian expression in integral} that
 \begin{equation}\label{eq:lemtf-1}
  \begin{split}
   \Delta^{\sigma}h(a)
   &=\int_{\R^n}\frac{h(x)-h(a)}{|x-a|^{n+2\sigma}}\ud x\\
   &=\int_{\Omega}\frac{x_n(1-|x'|^2)-a_n(1-|a'|^2)}{|x-a|^{n+2\sigma}}\ud x-\int_{\Omega^c}\frac{a_n(1-|a'|^2)}{|x-a|^{n+2\sigma}}\ud x\\
   &=:I-a_n II.
  \end{split}
 \end{equation}
Since $x_n(1-|x'|^2)-a_n(1-|a'|^2)=(x_n-a_n)(1-|x'|^2)+a_n(|a'|^2-|x'|^2)$, we divide the integral $I$ into
\[
 I_1:=\int_{\Omega}\frac{(x_n-a_n)(1-|x'|^2)}{|x-a|^{n+2\sigma}}\ud x
\]
and
\[
 a_n I_2:=a_n\int_{\Omega}\frac{(|a'|^2-|x'|^2)}{|x-a|^{n+2\sigma}}\ud x.
\]
By symmetry and that $0\leq a_n<1/8$,
\be\label{eq:lemtf-2}
\begin{split}
  I_1&=\int_{-1}^{-1+2a_n}\int_{|x'|<1}\frac{(x_n-a_n)(1-|x'|^2)}{|x-a|^{n+2\sigma}}\,\ud x'\ud x_n\\
     &\geq -C a_n.
 \end{split}
\ee
Using $|a'|^2-|x'|^2=-|x'-a'|^2+2a'(a'-x')$, we write
\[
 \begin{split}
  I_2&=\int_{\Omega}\frac{-|x'-a'|^2}{|x-a|^{n+2\sigma}}\ud x+\int_{\Omega}\frac{2a'\cdot(a'-x')}{|x-a|^{n+2\sigma}}\ud x\\
&=:I_3+I_4.
 \end{split}
\]
Direct computations give
\be\label{eq:lemtf-3}
  \begin{split}
  I_3&\geq -\int_{-2+a_n}^{2+a_n} \ud x_n \int_{|x'-a'|<2}\frac{-|x'-a'|^2}{|x-a|^{n+2\sigma}}\ud x'\\
&=-2\lim_{b\to 0^+}\int_{b}^{2} \ud y\int_{0}^{2}\frac{r^2 r^{n-2}}{(r^2+y^2)^{\frac{n+2\sigma}{2}}} \ud r\\
&=-2\lim_{b\to 0^+}\int_{b}^{2} y^{1-2\sigma} \ud y\int_{0}^{2/y}\frac{r^n}{(1+r^2)^{\frac{n+2\sigma}{2}}} \ud r\\
&=-2\lim_{b\to 0^+}\int_{b}^{2} y^{1-2\sigma} \ud y\left(\int_{1}^{2/y}\frac{r^n}{(1+r^2)^{\frac{n+2\sigma}{2}}}\ud r+\int_{0}^{1}\frac{r^n}{(1+r^2)^{\frac{n+2\sigma}{2}}} \ud r\right)\\
&\geq -2\lim_{b\to 0^+}\int_{b}^{2} y^{1-2\sigma} \ud y\left(\int_{1}^{2/y}r^{-2\sigma} \ud r +1\right)\\
&\geq -C.
 \end{split}
\ee

Next, we are going to show
\be\label{I4II} I_4-II\geq -C.
\ee
Let $D_0=(B'(0,1)\cap B'(2a', 1))$.  Since $a'=(a_1,0,\cdots,0)$, it follows from symmetry that
\[
 \int_{D_0\times (-1,1)}\frac{2a'\cdot(a'-x')}{|x-a|^{n+2\sigma}}\,\ud x'\ud x_n=0.
\]
Thus,
\[
 I_4=\int_{(B'(0,1)\setminus D_0)\times (-1,1)}\frac{2a'\cdot(a'-x')}{|x-a|^{n+2\sigma}}>0.
\]
Now we have two cases:

\medskip

\noindent \emph{Case 1.} if $|a'|\leq\frac{\sqrt{2}}{2}$, then it is easy to see that $II<C$ (the denominator is uniformly bounded). Hence, \eqref{I4II} holds.

\medskip

\noindent \emph{Case 2.} Suppose $|a'|>\frac{\sqrt{2}}{2}$. We have
\[
\begin{split}
 II&=\int_{\Omega^c\cap(B'(a',|a'|)\times (-1,1))}\frac{1-|a'|^2}{|x-a|^{n+2\sigma}}+\int_{\Omega^c\backslash (B'(a',|a'|)\times (-1,1))}\frac{1-|a'|^2}{|x-a|^{n+2\sigma}}\\
&\leq \int_{\Omega^c\cap(B'(a',|a'|)\times (-1,1))}\frac{(1-|a'|^2)}{|x-a|^{n+2\sigma}}+C\\
&=: II_1+C.
\end{split}
\]

Denote $D_1=\Big(B'(a', \sqrt{1-|a'|^2})\cap\{x_1<a_1\}\Big)\setminus D_0$, and $D_2=\Big(B'(a', \sqrt{1-|a'|^2})\cap\{x_1>a_1\}\Big)\setminus D_0$.

Note that for $x\in D_1$, we have $2|a'|(|a'|-x_1)\geq 1-|a'|^2-|x'-a'|^2.$
Therefore,
\[
 \int_{D_1\times (-1,1)}\frac{2a'\cdot(a'-x')}{|x-a|^{n+2\sigma}}- \int_{D_2\times (-1,1))}\frac{1-|a'|^2}{|x-a|^{n+2\sigma}}
\geq \int_{D_1\times (-1,1)}\frac{-|x'-a'|^2}{|x-a|^{n+2\sigma}}.
\]
Observe that there exists a positive integer $m$, which depends only on $n$ and $\sigma$, such that
\[
\begin{split}
 & m\int_{\Big(B'(0',1)\backslash B'(a', \sqrt{1-|a'|^2})\Big)\times (-1,1)}\frac{1-|a'|^2}{|x-a|^{n+2\sigma}} \\
& \geq \int_{\Big(B(a',|a'|)\backslash\big(B'(0',1)\cup B'(a', \sqrt{1-|a'|^2})\big)\Big)\times (-1,1)}\frac{1-|a'|^2}{|x-a|^{n+2\sigma}}.
\end{split}
\]
Also notice that for any $x\in B'(0',1)\backslash B'(a', \sqrt{1-|a'|^2})$, we have
\[
0\geq m(1-|a'|^2-|x'-a'|^2).
\]
Hence,
\[
\begin{split}
 & m\int_{\Big(B'(0',1)\backslash B'(a', \sqrt{1-|a'|^2})\Big)\times (-1,1)}\frac{|x'-a'|^2}{|x-a|^{n+2\sigma}} \\
& \geq \int_{\Big(B(a',|a'|)\backslash\big(B'(0',1)\cup B'(a', \sqrt{1-|a'|^2})\big)\Big)\times (-1,1)}\frac{1-|a'|^2}{|x-a|^{n+2\sigma}}.
\end{split}
\]
It follows that
\[
 \begin{split}
  &I_4-II\\
&\geq -C+\int_{D_1\times (-1,1)}\frac{2a'\cdot(a'-x')}{|x-a|^{n+2\sigma}}- \int_{D_2\times (-1,1)}\frac{1-|a'|^2}{|x-a|^{n+2\sigma}}\\
&\quad -\int_{\Big(B(a',|a'|)\backslash\big(B'(0',1)\cup B'(a', \sqrt{1-|a'|^2})\big)\Big)\times (-1,1)}\frac{1-|a'|^2}{|x-a|^{n+2\sigma}}\\
&\geq -C-m\int_{(B'(0',1)\backslash B'(a', \sqrt{1-|a'|^2}))\times (-1,1)}\frac{|x'-a'|^2}{|x-a|^{n+2\sigma}}+\int_{D_1\times (-1,1)}\frac{-|x'-a'|^2}{|x-a|^{n+2\sigma}}\\
&\geq -C-(m+1)I_3\\
&\geq -C.
 \end{split}
\]
 Therefore, \eqref{I4II} holds.

Finally, Lemma \ref{lemtf} follows from \eqref{eq:lemtf-1}, \eqref{eq:lemtf-2}, \eqref{eq:lemtf-3} and \eqref{I4II}.
\end{proof}

\begin{lem}\label{hopf}
Let $w(x,t)$ be as in Lemma \ref{wmp}. Suppose $w(x_0,0)>0$ for some $x_0\in \R^n_+$, then for any fixed $t\in (0,T]$, we have $\pa_{x_n}w(x',0,t)>0$.
\end{lem}

\begin{proof}
Let
\[
 g(x)=\begin{cases}
       -1,&\quad \mbox{in }B'(0,1)\times (-2,-1),\\
       1,&\quad \mbox{in }B'(0,1)\times (1,2),\\
       0,& \quad \mbox{otherwise},
      \end{cases}
\]
where $B'(0,1)$ denotes the $n-1$ dimensional unit ball centered at $0$. For any $x\in B'(0,1)\times (0,1/8)$, we have
\[
\begin{split}
 \Delta^\sigma g(x)&=\int_{\R^n}\frac{g(y)-g(x)}{|y-x|^{n+2\sigma}}\,\ud y\\&
=\int_{B'(0,1)\times (1,2)}\frac{1}{|y-x|^{n+2\sigma}}\,\ud y-\int_{B'(0,1)\times (-2,-1)}\frac{1}{|y-x|^{n+2\sigma}}\,\ud y\\&
=\int_{B'(0,1)\times (1,2)}\frac{1}{|y-(x',x_n)|^{n+2\sigma}}-\frac{1}{|y-(x',-x_n)|^{n+2\sigma}}\,\ud y\\&
=\int_{B'(0,1)\times (1,2)} \int_{0}^1-\frac{\ud }{\ud s}\left(\frac{1}{|y-x+2sx_ne_n|^{n+2\sigma}}\right)\,\ud s\,\ud y\\&
=(n+2\sigma)\int_{B'(0,1)\times (1,2)} \int_{0}^1 \frac{4(y_n-x_n)x_n+8sx_n^2}{|y-x+2sx_ne_n|^{n+2+2\sigma}}\,\ud s\,\ud y\\&
\geq c_1 x_n,
\end{split}
\]
where $c_1>0$ depends only on $n$ and $\sigma$.

For any fixed $t_0\in (0,T]$, define
\[
 H(x,t)=h(x)\Big(\frac{t_0^2}{1+t_0^2}-(t-t_0)^2\Big)+kg(x),
\]
where $h$ is as in Lemma \ref{lemtf}. We can choose a sufficiently large constant $k$ such that
\[
 H_t(x,t)\leq a(x,t)\Delta^\sigma H+b(x,t)H(x,t),
\]for all $x\in B'(0,1)\times (0,1/8)$ and $t\in (t_0-t_0/\sqrt{1+t_0^2},t_0]$.

It follows from Lemma \ref{smp} that $w(\cdot,t)>0$ in $\R^n_+$ for any fixed $t\in (0,T]$. Making a similar argument to the poof of Lemma \ref{smp},
we can show that there exists a small positive constant $\va$ such that
$w\geq \va H$ for all $t\in(0,t_0]$. Therefore, $\pa_{x_n} w(x',0,t_0)>0$ and Lemma \ref{hopf} follows immediately.
\end{proof}

Now we apply the above strong maximum principle and Hopf lemma to fractional Yamabe flow equations.

Suppose that $v$ is a positive smooth solution of \eqref{Y2} in $\Sn\times [0,T]$. Hence $$u(x,t)=\left(\frac{2}{1+|x|^2}\right)^{\frac{n-2\sigma}{2}} v(F(x),t)$$ satisfies \eqref{Y3}.
For a given real number $\lda$, define
\[
 \Sigma_\lda=\{x=(x',x_n):x_n\geq \lda\},
\]
and let $x^\lda=(x',2\lda -x_n)$ and $u_\lda(x,t)=u(x^\lda,t)$.
It is clear that $u_{\lda}$ also satisfies \eqref{Y3}.

\begin{prop}\label{weak maximum principle}
 Suppose that $u(x,0)-u_{\lda}(x,0)\geq 0$ in $\Sigma_\lda$, 
then for any fixed $t\in (0,T]$, we have $u(x,t)-u_\lda(x,t)\geq0$ in $\Sigma_\lda$.
\end{prop}

\begin{proof}
 Let $w(x,t)=u(x,t)-u_{\lda}(x,t)$. Then $w$ satisfies
 \be\label{eq:diff w}
 w_{t}=a(x,t)\Delta^{\sigma} w + b(x,t) w
 \ee
 where $a(x,t)=\frac{1}{Nu^{N-1}}$ and $b(x,t)=\frac{(1-N)(-\Delta)^{\sigma}u_{\lda}}{N}\int_0^1\frac{1}{(\tau u+(1-\tau) u_{\lda})^N}d \tau+\frac{r_{\sigma}^g}{N}$ is bounded. Note that $w(x',x_n+\lda,t)$ satisfies all the conditions in Lemma \ref{wmp}. Thus Proposition \ref{weak maximum principle} follows from Lemma \ref{wmp}.
\end{proof}

\begin{prop}\label{w-s max principle}
Assume the conditions in Proposition \ref{weak maximum principle}, then for any fixed $t\in (0,T]$, we have $u(x,t)-u_\lda(x,t)>0$ or $u(x,t)-u_\lda(x,t)\equiv0$ in $\Sigma_\lda$.
\end{prop}

\begin{proof}
It follows from Proposition \ref{weak maximum principle} and Lemma \ref{w-s mp}.
\end{proof}

\begin{prop}\label{strong maximum principle}
Assume the conditions in Proposition \ref{weak maximum principle}. In addition, we suppose that $u(x_0,0)-u_{\lda}(x_0,0)> 0$ for some $x_0\in \Sigma_{\lda}$,
then for any fixed $t\in (0,T]$, we have $u(x,t)-u_\lda(x,t)>0$ in $\Sigma_\lda$ and $\pa_{x_n}u(x',\lda,t)>0$.
\end{prop}

\begin{proof}
It follows from Proposition \ref{weak maximum principle}, Lemma \ref{smp} and Lemma \ref{hopf}.
\end{proof}

\section{Harnack inequality for a fractional Yamabe flow}\label{harnack}

Based on the results proved in the previous section, we are going to establish the following Harnack inequality.

\begin{thm}\label{thm:harnack} Let $v$ be a $C^{3,1}$ positive function on $\Sn\times [0,T^*)$ and satisfy
\[
\frac{\pa v^N}{\pa t}=-P_\sigma(v)+b(t)v^N,\quad \mbox{on }\Sn\times (0,T^*),
\]
where $b(t)\in C([0,T^*))$.
Then there exists a positive constant $C>0$ depending only on $n,\sigma, \inf_{\Sn} v(\cdot,0)$ and $\|v(\cdot, 0)\|_{C^{3}(\Sn)}$ such that
\[
\max_{\Sn} v(\cdot,t) \leq C \min_{\Sn} v(\cdot,t),
\]for any fixed $t\in (0,T^*)$.
\end{thm}

\begin{proof}
As mentioned in the introduction, the idea of this proof is essentially due to Ye \cite{Y}. We will show that
\[
\sup_{\Sn}\frac{\abs{\nabla_{\Sn} v}}{\abs{v}} \leq C \quad\text{   for all } s\in (0,T^*).
\]
Let $q_0\in\Sn$. Without loss of generality, we may assume that $q_0$ is the north pole. Consider the inverse of the stereographic projection from the north pole $F: \R^n\to\Sn$:
\[
F(x_1,\cdots, x_n)=\left(\frac{2x}{1+x^2},\frac{x^2-1}{x^2+1}\right).
\]
We also denote $G: \R^n\to\Sn$ as the inverse of the stereographic projection from the south pole, namely $G(x)=F(x/\abs{x}^2)$.
Let
\[
u(x,s)=\left(\frac{2}{1+\abs{x}^2}\right)^{\frac{n-2\sigma}{2}}v(F(x),s),\quad \bar u(x,s)=\left(\frac{2}{1+\abs{x}^2}\right)^{\frac{n-2\sigma}{2}}v(G(x),s).
\]
Then $u,\bar u\in C^{3,1}(\R^n\times [0,T^*))$ and both satisfy
\be\label{eq:in harnack}
\frac{\pa u^N}{\pa t}=\Delta^{\sigma}u+b(t)u^N,\quad \mbox{on }\R^n\times [0,T^*).
\ee
$u(\cdot, s)$ has a Taylor expansion ``at infinity" of the form
\[
u(x,s)=\frac{2^{(n-2\sigma)/2}}{\abs{x}^{n-2\sigma}}\left(a_0+\frac{a_i x_i}{\abs{x}^2}+\left(a_{ij}-\frac{n-2\sigma}{2}\delta_{ij}\right)\frac{x_i x_j}{\abs{x}^4}+O(\abs{x}^{-3})\right).
\]
Similarly, the partial derivatives of $u(\cdot, t)$ have Taylor expansions ``at infinity" of the form
\[
\begin{split}
\frac{\pa u}{\pa x_i}(x,s)
&=2^{\frac{n-2\sigma}{2}}\left(-\frac{n-2\sigma}{\abs x^{n-2\sigma+2}}x_i\left(a_0+\frac{a_jx_j}{\abs x^2}\right)+\frac{a_i}{\abs x^{n-2\sigma+2}}-\frac{2x_ia_jx_j}{\abs x^{n-2\sigma+4}}\right)\\
&\quad+O(\abs x^{-(n-2\sigma+3)}).
\end{split}
\]
Here
\begin{equation*}
\begin{split}
a_0(s) &= v(q_0,s),\\
a_i(s) &=\frac{\pa (v(\cdot,s)\circ G)}{\pa x_i}(0),\\
a_{ij}(s)&=\frac{\pa^2 (v(\cdot,s)\circ G)}{2\pa x_i x_j}(0).
\end{split}
\end{equation*}
Let $y_{i}(s)=(n-2\sigma)^{-1} a_i(s)/a_0(s)$, and $y(s)=(y_1(s),\cdots, y_n(s))$. Then
\begin{equation}\label{expansion}
u(x+y,s)=\frac{2^{\frac{n-2\sigma}{2}}}{\abs x^{n-2\sigma}}\left(a_0+\frac{\tilde a_{ij}x_ix_j}{\abs x^4}+o(\abs x^{-2})\right)
\end{equation}
and
\begin{equation}\label{expansion of derivative}
\frac{\pa u}{\pa x_i}(x+y,s)=-\frac{(n-2\sigma)a_0x_i}{\abs x^{n-2\sigma+2}}+O(\abs x^{-(n-2\sigma+3)}),
\end{equation}
where $\tilde a_{ij}=a_{ij}-\frac{n-2\sigma}{2}\delta_{ij}-\frac{a_ia_i}{a_0}$.
We only need to show that there exists a positive constant $C$ depending only on $n,\sigma, \inf_{\Sn} v(\cdot,0)$ and $\|v(\cdot, 0)\|_{C^{3}(\Sn)}$ such that
\[
\abs {y(s)}\leq C\quad \text{for all }0\leq s<T^*.
\]
Fix $T\in (0,T^*)$. After a rotation and a reflection, we may assume that $y_n(T)=\max_i \abs{y_i(T)}$.
From the Taylor expansions of $u$ and $\nabla u$ for $s=0$, we see that (e.g., Lemma 4.2 in \cite{GNN}) there exists a $\lda_0>0$, which depends only on $n,\sigma, \inf_{\Sn} v(\cdot,0)$ and $\|v(\cdot, 0)\|_{C^{3}(\Sn)}$, such that for any $\lda>\lda_0$,
\[
u(x,0)>u(x^{\lda},0) \quad\text{for }x_n<\lda,
\]
where $x^{\lda}=(x_1,\cdots,x_{n-1},2\lda-x_n)$.
Denote $u^{\lda}(x,s)=u(x^{\lda},s)$. By Proposition \ref{strong maximum principle}, we have
\begin{equation}\label{strictly bigger}
u(x,s)>u^{\lda}(x,s)\quad \text{for all } s\in [0,T],\ x_n<\lda,\ \lda\geq \lda_0.
\end{equation}
We claim that
\[
\max\limits_{0\leq s \leq T} y_n(s)<\lda_0.
\]
If not, there exists $\bar s\in (0,T]$ such that $y_n(\bar s)=\max_{0\leq s \leq T} y_n(s)\geq\lda_0$. Thus, we can set $\lda=y_n(\bar s)$ in \eqref{strictly bigger}, namely,
\[
u(x,s)>u^{\lda}(x,s)\quad \text{for all } s\in [0,T],\ x_n<\lda=y_n(\bar s).
\]
Let $\tilde u(x,s)= u(x+y_n(\bar s),s)$, then
\[
\tilde u(x',x_n,s)>\tilde u(x',-x_n,s) \quad \text{for all } s\in [0,T],\ x_n<0.
\]
Let $\tilde u_1 (x,s)=\frac{1}{\abs x^{n-2\sigma}}\tilde u (\frac{x}{\abs x^2},s)$. Then $\tilde u_1(x',x_n,s)$ and $\tilde u_1(x',-x_n,s)$ satisfy \eqref{eq:in harnack} and
\[
\tilde u_1(x',x_n,s)>\tilde u_1(x',-x_n,s) \quad \text{for all } s\in [0,T],\ x_n<0.
\]
By Proposition \ref{strong maximum principle},
\[
\left.\frac{\pa (\tilde u_1(x',x_n,s)-\tilde u_1(x',-x_n,s))}{\pa x_n}\right|_{(x,s)=(0,\bar s)}<0,
\]
i.e., $(\pa \tilde u_1/\pa x_n)(0,\bar s)<0$. This contradicts \eqref{expansion}. Hence, $\max\limits_{0\leq s \leq T} y_n(s)<\lda_0$, which implies $y_n(T)<\lda_0$. Since $\lda_0$ is independent of $s$, we have $\abs {y(s)}\leq \lda_0\quad \text{for all }0\leq s<T^*.$ Moreover, $\lda_0$ is independent of the choice of $q_0$, and we conclude that
\[
\sup_{\Sn}\frac{\abs{\nabla_{\Sn} v}}{\abs{v}} \leq C \quad\text{   for all } s\in (0,T^*).
\]
For each $t$, integrating the above inequality along a shortest geodesic between a maximum point and a minimum point of $v(\cdot,t)$ yields
\[
\max_{\Sn} v(\cdot,t) \leq C \min_{\Sn} v(\cdot,t),
\]
where $C$ depends only on $n,\sigma, \inf_{\Sn} v(\cdot,0)$ and $\|v(\cdot, 0)\|_{C^{3}(\Sn)}$.
\end{proof}

\section{Existence and convergence of a fractional Yamabe flow}\label{section: existence and convergence}

\label{yambe}

\subsection{Schauder estimates}
For an open set $\Omega\subset \R^n$ and $\gamma\in (0,1)$,  $C^\gamma(\om)$ denotes the standard H\"older space over $\om$, with the norm 
\[
|v|_{\gamma; \Omega}:=|v|_{0;\om}+[v]_{\gamma;\om}:=\sup_{\om}|v(\cdot)|+\sup_{x_1\neq x_2, x_1,x_2\in \om}\frac{|u(x_1)-u(x_2)|}{|x_1-x_2|^{\gamma}}.
\]
For simplicity, we use $C^{\gamma}(\om)$ to denote $C^{[\gamma],\gamma-[\gamma]}(\om)$ when $1<\gamma\notin \mathbb{N}$ (the set of positive integers), where $[\gamma]$ is the integer part of $\gamma$. 
Since the operator $\pa_t+(-\Delta)^{\sigma}$ is invariant under the scaling $(x,t)\to (c x, c^{2\sigma}t)$ with $c>0$,
we introduce the fractional parabolic distance as
\[
 \rho(X_1,X_2)=(|x_1-x_2|^2+|t_1-t_2|^{1/\sigma})^{1/2},
\]
where  $X_1=(x_1,t_2), X_2=(x_2,t_2)\in \R^{n+1}$. For a measurable function $u$ defined in a Borel set $Q\subset \R^{n+1}$ and $0<\al<\min(1,2\sigma)$, we define
\[
[u]_{\al, \frac{\al}{2\sigma};Q}=\sup_{X_1\neq X_2, X_1,X_2\in Q}\frac{|u(X_1)-u(X_2)|}{\rho(X_1,X_2)^{\al}},
\] 
and
\[
|u|_{\al,\frac{\al}{2\sigma};Q}=|u|_{0;Q}+[u]_{\al,\frac{\al}{2\sigma};Q},
\]
where $|u|_{0;Q}=\sup_{X\in Q}|u(X)|$. We denote $C^{\al, \frac{\al}{2\sigma}}(Q)$ as the space of all measurable functions $u$ for which $ |u|_{\al, \frac{\al}{2\sigma};Q}<\infty$. Let $Q_T=\R^n\times (0,T]$, $T\in (0,\infty)$. For $2\sigma+\al\notin \mathbb{N}$ and $0<\al<\min(1,2\sigma)$, we say $u\in \mathcal C^{2\sigma+\al,1+\frac{\al}{2\sigma}}(Q_T)$ if 
\[
[u]_{2\sigma+\al,1+\frac{\al}{2\sigma};Q_T}:=[u_t]_{\al, \frac{\al}{2\sigma};Q_T}+[(-\Delta)^{\sigma}u]_{\al, \frac{\al}{2\sigma};Q_T}<\infty
\]
and
\[
|u|_{2\sigma+\al,1+\frac{\al}{2\sigma};Q_T}:=|u|_{0;Q_T}+|u_t|_{0,Q_T}+|(-\Delta)^{\sigma}u|_{0;Q_T}+[u]_{2\sigma+\al,1+\frac{\al}{2\sigma};Q_T}<\infty.
\]
Then $\mathcal C^{2\sigma+\al,1+\frac{\al}{2\sigma}}(Q_T)$ is a Banach space equipped with the norm $|\cdot|_{2\sigma+\al,1+\frac{\al}{2\sigma};Q_T}$.

Consider the following Cauchy problem
\be\label{lpde}
\begin{cases}
 a(x,t)u_t+(-\Delta)^\sigma u+b(x,t)u=f(x,t),&\quad \mbox{in }Q_T,\\
u(x,0)=u_0(x),&\quad \mbox{in }\R^n,
\end{cases}
\ee
where $\lda^{-1}\leq a(x,t)\leq \lda$ for some constant $\lda\geq 1$.

\begin{lem}\label{heat maximum principle}
Suppose $b(x,t)$ is bounded in $Q_1$. Let $u\in \mathcal C^{2\sigma+\al,1+\frac{\al}{2\sigma}}(Q_1)$ satisfy
\[
\begin{cases}
 a(x,t)u_t+(-\Delta)^\sigma u+b(x,t)u\leq 0,&\quad \mbox{in }Q_1,\\
u(x,0)\leq 0,&\quad \mbox{in }\R^n,
\end{cases}
\]
then $u\leq 0$ in $Q_1$.
\end{lem}
\begin{proof}
Without loss of generality we may assume that $b(x,t)\geq 1$ as before. Let $\eta(x)$ be a smooth cut-off function supported in $B_2\subset \R^n$ and equal to $1$ in $B_1$. Let $\eta_{R}(\cdot)=\eta(\cdot/R)$ and $v=\eta_R u$. Then
\[
a v_t+(-\Delta)^\sigma v+b(x,t) v\leq  \langle u,\eta_R\rangle+u(-\Delta)^\sigma \eta_R,
\]
where
\be\label{fractional inner product}
 \langle u,\eta\rangle=c(n,\sigma)\int_{\R^n}\frac{(u(x,t)-u(y,t))(\eta(x,t)-\eta(y,t))}{|x-y|^{n+2\sigma}}\,\ud y.
\ee
If $u$ is positive somewhere in $Q_1$, then we can choose $R$ as large as we want such that $v$ attains its positive maximum value in $Q_1$ at $(x_0,t_0)\in B_R\times(0,1]$. It is clear that $a(x_0,t_0) v_t(x_0,t_0)+(-\Delta)^\sigma v(x_0,t_0)\geq 0$. Since $b\geq 1$, we have
\[
\sup_{B_R\times(0,1]}u\leq v(x_0,t_0)\leq \sup_{Q_1}|\langle u,\eta_R\rangle+u(-\Delta)^\sigma \eta_R|\to 0\quad\mbox{as }R\to\infty.
\]
This finishes the proof of this Lemma.
\end{proof}

\begin{prop}\label{schauder estimates} Let $0<\al<\min(1,2\sigma)$ such that $2\sigma+\al$ is not an integer. Suppose that $a(x,t)$, $b(x,t)$, $f(x,t)$ $\in C^{\al,\frac{\al}{2\sigma}}(Q_1)$ and $u_0(x)\in C^{2\sigma+\al}(\R^n)$. Then there exists a unique solution $u\in \mathcal C^{2\sigma+\al,1+\frac{\al}{2\sigma}}(Q_1)$ of \eqref{lpde}. Moreover, there exists a constant $C>0$ depending only on $n,\sigma, \lda, \al, |a|_{\al,\frac{\al}{2\sigma};Q_1}$ and $|b|_{\al,\frac{\al}{2\sigma};Q_1}$ such that 
\be\label{sch}
 |u|_{2\sigma+\al, 1+\frac{\al}{2\sigma};Q_1}\leq C(|u_0|_{2\sigma+\al; \R^n}+|f|_{\al,\frac{\al}{2\sigma};Q_1}).
\ee
\end{prop}

\begin{proof}
By Lemma \ref{heat maximum principle}, there exists $C>0$ depending only on $\lda, |b|_{L^{\infty}(Q_1)}$ such that
\be\label{infinity norm controlled}
|u|_{0; Q_1}\leq C(|u_0|_{0; \R^n}+|f|_{0; Q_1}).
\ee
Then the uniqueness of solutions of \eqref{lpde} follows immediately. In the following, we will show the a priori estimate \eqref{sch}. By \eqref{infinity norm controlled} and some interpolation inequalities in Lemma \ref{interpolation inequality}, we only need to show, instead of \eqref{sch}, \be\label{seminorm sch}
[u]_{2\sigma+\al, 1+\frac{\al}{2\sigma};Q_1}\leq C(|u_0|_{2\sigma+\al; \R^n}+|f|_{\al,\frac{\al}{2\sigma};Q_1}).
\ee
First of all, \eqref{seminorm sch} holds provided $a=1, b=0$ (see, e.g., \cite{MR}),
and it can be easily extended to the case that $a$ is a positive constant.
For the general case, we use the ``freezing coefficients'' method (see, e.g., \cite{Kr}).

Fix a small $\delta>0$, which will be specified later. We can find two points $X_1, X_2\in Q_1$ such that
\[
 \frac{|u_t(X_1)-u_t(X_2)|}{\rho(X_1,X_2)^\al} \geq \frac{1}{2}[u_t]_{\al, \frac{\al}{2\sigma};Q_1}.
\]
If $\rho(X_1,X_2)>\delta$, then
\[
 [u_t]_{\al, \frac{\al}{2\sigma};Q_1}\leq 4\delta^{-\al}|u_t|_{0;Q_1}.
\]
It follows from Lemma \ref{interpolation inequality} that, for any small $\va_0>0$,
\be\label{sch1}
[u_t]_{\al, \frac{\al}{2\sigma};Q_1}\leq \va_0[u]_{2\sigma+\al, 1+\frac{\al}{2\sigma};Q_1}+C_0|u|_{0;Q_1},
\ee
where $C_0>0$ depends on $n,\sigma, \al,\va_0, \delta$.

If $\rho(X_1,X_2)\leq \delta$, take a cut-off function $\eta(X)\in C^\infty(\R^{n+1})$
such that $\eta(X)=1$ for $\rho(X,X_1)\leq \delta$, $\eta(X)=0$ for $\rho(X,X_1)\geq 2\delta$.
By the estimates of solutions of \eqref{lpde} with $a$ being a positive constant and $b\equiv 0$, we have
\[
\begin{split}
 [u_t]_{\al, \frac{\al}{2\sigma};Q_1}&\leq 2 \frac{|u_t(X_1)-u_t(X_2)|}{\rho(X_1,X_2)^\al}\leq 2[u\eta]_{2\sigma+\al,1+\frac{\al}{2\sigma};Q_1}\\&
\leq C_1(|a(X_1)(u\eta)_t+(-\Delta)^\sigma(u\eta)|_{\al,\frac{\al}{2\sigma};Q_1}+|u_0\eta|_{2\sigma+\al;\R^n}+|u\eta|_{0;Q_1}),
\end{split}
\]
where $C_1>0$ is independent of $\delta$.
Note that
\[
 \begin{split}
  & a(X_1)(u\eta)_t+(-\Delta)^\sigma(u\eta)\\&
=\eta(a(X)u_t+(-\Delta)^\sigma u)+\eta(a(X_1)-a(X))u_t+a(X_1)u\eta_t-\langle u,\eta\rangle+u(-\Delta)^\sigma\eta \\&
=\eta(f-bu)+\eta(a(X_1)-a(X))u_t+a(X_1)u\eta_t-\langle u,\eta\rangle+u(-\Delta)^\sigma\eta,
 \end{split}
\]
where
$ \langle u,\eta\rangle $ is defined in \eqref{fractional inner product}. 
Since $|\eta(X)(a(X_1)-a(X))|\leq [a]_{\al,\frac{\al}{2\sigma};Q_1}\delta^\al$, making use of Lemma \ref{interpolation inequality}
again, we have
\be\label{ut estimate before inter}
\begin{split}
  [u_t]_{\al, \frac{\al}{2\sigma};Q_1}&\leq C_1\delta^\al[u]_{2\sigma+\al,1+\frac{\al}{2\sigma};Q_1}+C(\delta)(|u|_{0;Q_1}+|f|_{\al,\frac{\al}{2\sigma};Q_1})\\
  &\quad +C_1|\langle u,\eta\rangle|_{\al,\frac{\al}{2\sigma};Q_1}+C_1|u_0\eta|_{2\sigma+\alpha;\R^n}.
\end{split}
\ee
Hence, from \eqref{estimate of inner product} in Lemma \ref{lemma of inner product}, \eqref{ut estimate before inter} and \eqref{sch1}, we can conclude that
\be\label{sch2}
 [u_t]_{\al, \frac{\al}{2\sigma};Q_1}\leq (C_1\delta^\al+\va_0)[u]_{2\sigma+\al,1+\frac{\al}{2\sigma};Q_1}+C(\delta)(|u|_{0,Q_1}+|f|_{\al,\frac{\al}{2\sigma};Q_1}+|u_0|_{2\sigma+\alpha;\R^n}).
\ee
Since
\[
 u_t+(-\Delta)^\sigma u=(1-a)u_t-bu+f,
\]
we see that
\be\label{para to elliptic}
 [u]_{2\sigma+\al,1+\frac{\al}{2\sigma};Q_1}\leq C([u_t]_{\al, \frac{\al}{2\sigma};Q_1}+|u|_{0;Q_1}+|f|_{\al,\frac{\al}{2\sigma};Q_1}+|u_0|_{2\sigma+\alpha;\R^n},),
\ee
where $C>0$ depending only on $n,\sigma, \lda, \al$, $\|a\|_{\al,\frac{\al}{2\sigma};Q_1}$ and $\|a, b\|_{\al,\frac{\al}{2\sigma};Q_1}$.
Then \eqref{sch} follows from \eqref{infinity norm controlled}, \eqref{para to elliptic} and \eqref{sch2} by choosing sufficiently small $\delta$ and $\va_0$. 

Finally, the existence of solutions of \eqref{lpde} follows from standard continuity method.
\end{proof}

\begin{rem} 
 Cauchy problems for non-local operators and pseudo-differential operators in different spaces have been studied, e.g., in \cite{Ko}, \cite{MP1}, \cite{MP2}, \cite{MP3} and references therein. 
 \end{rem}

\begin{rem}\label{Ob1}
Observe that in the proof of the above proposition the only place we use the uniform lower and upper bounds of $a(x)$ is that at $X_1$, that is  $\frac{1}{\lda}\leq a(X_1)\leq \lda$. This observation will be used in the proof of Proposition \ref{schsn12}.
\end{rem}

\begin{rem}
One can also obtain the estimates in $Q_T$ by considering the scaled function $\tilde u(x,t):=u(T^{1/2\sigma}x, Tt)$.
\end{rem}

For $\gamma\in (0,1)$,  $C^\gamma(\Sn)$ denotes the standard H\"older space over $\Sn$, with norm 
\[
|v|_{\gamma; \Sn}:=|v|_{0;\Sn}+[v]_{\gamma;\Sn}:=\sup_{\Sn}|v(\cdot)|+\sup_{\xi_1\neq \xi_2, \xi_1,\xi_2\in \Sn}\frac{|v(\xi_1)-u(\xi_2)|}{|\xi_1-\xi_2|^{\gamma}},
\]
where  $|\xi_1-\xi_2|$ is understood as the Euclidean distance from $\xi_1$ to $\xi_2$ in $\R^{n+1}$.
For simplicity, we use $C^{\gamma}(\Sn)$ to denote $C^{[\gamma],\gamma-[\gamma]}(\Sn)$ when $1<\gamma\notin \mathbb{N}$, where $[\gamma]$ is the integer part of $\gamma$. 
For $Y_1=(\xi_1,t_1), Y_2=(\xi_2,t_2)\in \Sn\times(0,\infty)$ we denote
\[
 \rho(Y_1,Y_2)=(|\xi_1-\xi_2|^2+|t_1-t_2|^{1/\sigma})^{1/2}.
\]
 We still assume that $0<\al<\min(1,2\sigma)$. Let $\mathcal Q_T=\Sn\times(0,T]$ for $T>0$ . We say $v\in C^{\al, \frac{\al}{2\sigma}}(\mathcal Q_T)$ if
\[
|v|_{\al,\frac{\al}{2\sigma};\mathcal Q_T}=|v|_{0;\mathcal Q_T}+[v]_{\al,\frac{\al}{2\sigma};\mathcal Q_T}:=\sup_{Y\in \mathcal Q_T}v(Y)+\sup_{Y_1\neq Y_2, Y_1,Y_2\in \mathcal Q_T}\frac{|v(Y_1)-u(Y_2)|}{\rho(Y_1,Y_2)^{\al}}<\infty,
\]
and $v\in \mathcal C^{2\sigma+\al,1+\frac{\al}{2\sigma}}(\mathcal Q_T)$ if 
\[
[v]_{2\sigma+\al,1+\frac{\al}{2\sigma};\mathcal Q_T}:=[v_t]_{\al, \frac{\al}{2\sigma};\mathcal Q_T}+[P_{\sigma}(v)]_{\al, \frac{\al}{2\sigma};\mathcal Q_T}<\infty
\]
and
\[
|v|_{2\sigma+\al,1+\frac{\al}{2\sigma};\mathcal Q_T}:=|v|_{0;\mathcal Q_T}+|v_t|_{0,\mathcal Q_T}+|P_{\sigma}(v)|_{0;\mathcal Q_T}+[v]_{2\sigma+\al,1+\frac{\al}{2\sigma};\mathcal Q_T}<\infty.
\]
Then $\mathcal C^{2\sigma+\al,1+\frac{\al}{2\sigma}}(\mathcal Q_T)$ is a Banach space equipped with the norm $|\cdot|_{2\sigma+\al,1+\frac{\al}{2\sigma};\mathcal Q_T}$.

\begin{prop}\label{schsn12} Let $0<\al<\min(1,2\sigma)$ such that $2\sigma+\al$ is not an integer. Let $a(\xi,t)$, $b(\xi,t)$, $f(\xi,t)$ $\in C^{\al,\frac{\al}{2\sigma}}(\mathcal Q_1)$, $v_0\in C^{2\sigma+\al}(\Sn)$ and $\lda^{-1} \leq a(\xi,t)\leq \lda$ for some $\lda \geq 1$.
 Then there exists a unique function $v\in \mathcal C^{2\sigma+\al, 1+\frac{\al}{2\sigma}}$$(\mathcal Q_1)$ such that
\be \label{Dir}
\begin{cases}
 a v_t+P_\sigma(v)+bv=f,&\quad \mbox{in } \mathcal Q_1, \\
v(y,0)=v_0(y),&\quad \mbox{in } \Sn.
\end{cases}
\ee 
Moreover, there exists a constant $C$ depending only on $n,\sigma, \lda, \al, |a|_{\al,\frac{\al}{2\sigma};\mathcal Q_1}$ and $|b|_{\al,\frac{\al}{2\sigma};\mathcal Q_1}$ such that
\be\label{sch13}
 |v|_{2\sigma+\al, 1+\frac{\al}{2\sigma};\mathcal Q_1}\leq C(|v_0|_{2\sigma+\al; \Sn}+|f|_{\al,\frac{\al}{2\sigma};\mathcal Q_1}).
\ee
\end{prop}

\begin{proof}
Uniqueness of solutions of \eqref{Dir} follows from maximum principles. We only need to show a priori estimate \eqref{sch13},
from which the existence of solution of \eqref{Dir} follows by the standard continuity method. 

Choose $Y_1=(\xi_1,t_1),Y_2=(\xi_2,t_2) \in \Sn\times (0,T)$ such that  
\be\label{sphere starting pt}
 \frac{|v_t(Y_1)-v_t(Y_2)|}{\rho(Y_1,Y_2)^\al}\geq \frac{1}{2}[v_t]_{\al,\frac{\al}{2\sigma};\mathcal Q_1}.
\ee
Without loss of generality we may assume that $\xi_1,\xi_2$ are on the south hemisphere. Let $F(x)$ be the inverse of stereographic projection from the north pole and  
\[
u(x,t)=\left(\frac{2}{1+|x|^2}\right)^{\frac{n-2\sigma}{2}} v(F(x),t).
\]
There exist $x_1, x_2\in B(0,1)$ such that $Y_1=(F(x_1),t_1), Y_2=(F(x_2),t_2)$. We denote $ X_1=(x_1,t_1), X_2=(x_2,t_2)$. By \eqref{sphere starting pt} there exists a constant $C$ depending only $n,\sigma,\al$ such that
\[
[u_t]_{\alpha,\frac{\al}{2\sigma};Q_1}\leq C|v_t|_{0,\mathcal Q_1}+C |u_t|_{0,Q_1}+C\frac{|u_t(X_1)-u_t(X_2)|}{\rho(X_1,X_2)^\al}.
\] 
 Note that $u$ satisfies \eqref{lpde} with $a, b,f$ replaced by 
\[
 \left(\frac{2}{1+|x|^2}\right)^{2\sigma}a(F(x),t), \ \left(\frac{2}{1+|x|^2}\right)^{2\sigma}b(F(x),t), \ \left(\frac{2}{1+|x|^2}\right)^{\frac{n+2\sigma}{2}}f(F(x),t).
\]
 In view of Remark \ref{Ob1} and the arguments in the proof of Proposition \ref{schauder estimates}, we conclude that 
\[
 [u]_{2\sigma+\al, 1+\frac{\al}{2\sigma}; Q_1}\leq C(|v_0|_{2\sigma+\al;\Sn}+|v|_{0;\mathcal Q_1}+|v_t|_{0;\mathcal Q_1}+|f|_{\al, \frac{\al}{2\sigma};\mathcal Q_1}).  
\]
Hence, together with \eqref{sphere starting pt} and interpolation inequalities in Lemma \ref{interpolation inequality on Sn}, we have
\be\label{vt estimate}
 [v_t]_{\al,\frac{\al}{2\sigma};\mathcal Q_1}\leq C(|v_0|_{2\sigma+\al;\Sn}+|v|_{0;\mathcal Q_1}+|f|_{\al,\frac{\al}{2\sigma};\mathcal Q_1}).
\ee
It follows from the maximum principle that $|v|_{0;\mathcal Q_1}\leq C(|v_0|_{2\sigma+\al;\Sn}+|f|_{\al,\frac{\al}{2\sigma};\mathcal Q_1})$. Hence \eqref{sch13} follows from \eqref{vt estimate}, \eqref{Dir} and some inequalities in Lemma \ref{interpolation inequality on Sn}.
\end{proof}

\begin{cor}\label{interior schsn} 
Let $0<\al<\min(1,2\sigma)$ such that $2\sigma+\al$ is not an integer. Let $a(\xi,t)$, $b(\xi,t)$, $f(\xi,t)$ $\in C^{\al,\frac{\al}{2\sigma}}(\mathcal Q_3)$,  $\lda^{-1} \leq a(\xi,t)\leq \lda$ for some $\lda \geq 1$. 
Suppose that $v\in \mathcal C^{2\sigma+\al, 1+\frac{\al}{2\sigma}}$$(\mathcal Q_3)$ satisfies
\[
 a v_t+P_\sigma(v)+bv=f,\quad \mbox{in } \mathcal Q_3.
\]
Then there exists a constant $C$ depending only on $n,\sigma, \lda, \al, |a|_{\al,\frac{\al}{2\sigma};\Sn\times[1,3]}$ and $|b|_{\al,\frac{\al}{2\sigma};\Sn\times [1,3]}$ such that
\[
 |v|_{2\sigma+\al, 1+\frac{\al}{2\sigma};\Sn\times [2,3]}\leq C(|v|_{\al,\frac{\al}{2\sigma}; \Sn\times [1,3]}+|f|_{\al,\frac{\al}{2\sigma};\Sn\times [1,3]}).
\]
\end{cor}
\begin{proof}
Let $\eta(t)$ be a smooth cut-off function defined on $\R$ such that $\eta(t)=0$ when $t\leq 4/3$ and $\eta(t)=1$ when $t\geq 5/3$. Then $\tilde v:=\eta v$ satisfies
\[
\begin{cases}
 a \tilde v_t+P_\sigma(\tilde v)+b\tilde v=f\eta+av\eta_t,&\quad \mbox{in } \Sn\times[1,3], \\
 \tilde v(\cdot, 1)=0.
\end{cases}
\]
The Corollary follows immediately from Proposition \ref{schsn12}.
\end{proof}

\subsection{Short time existence}\label{short time existence}

\begin{prop}\label{STE}
Let $0<\al<\min(1,2\sigma)$ such that $2\sigma+\al$ is not an integer. Let $v_0\in C^{2\sigma+\al}(\Sn)$ and $v_0>0$ in $\Sn$. Then there exists a small positive constant $T_*$  depending only on $n, \sigma, \al, \inf_{\Sn}v_0, |v_0|_{2\sigma+\al;\Sn}$ and a unique positive solution $v\in \mathcal C^{2\sigma+\al, 1+\frac{\al}{2\sigma}}(\Sn\times [0,T_*])$ of \eqref{Y2} in $\Sn\times (0, T_*]$ with $v(\cdot,0)=v_0$. Furthermore, $v$ is smooth in $\Sn\times(0,T_*)$.
\end{prop}

\begin{proof} By a scaling argument in the time variable, we only need to show the short time existence of
\[
\begin{cases}
\frac{\pa v^N}{\pa t}=-P_{\sigma}(v),\\
v(\cdot, 0)=v_0.
\end{cases}
\]

We shall use the Implicit Function Theorem.
By Proposition \ref{schsn12}, there exists a function $w\in \mathcal C^{2\sigma+\al, 1+\frac{\al}{2\sigma}}(\Sn\times (0,1])$ such that
\[
 \begin{cases}
  N v_0^{N-1} w_t=-P_\sigma(w),& \quad \mbox{in }\Sn\times (0,1],\\
w(\cdot, 0)=v_0,
 \end{cases}
\]
and for any small positive constant $\va_0$, we have $\|w(\cdot,t)-v_0\|_{C^{2\sigma+\al}(\Sn)}\leq \va_0$ provided $t\leq T_{\va_0}$. Here $T_{\va_0}$ is a positive constant
depending on $\va_0$. Hence, we may assume that $w>0$ in $\Sn$.

Denote
\[
 \mathscr{X}=\{\varphi\in \mathcal C^{2\sigma+\al, 1+\frac{\al}{2\sigma}}(\Sn\times (0,T_{\va_0}]): \varphi(\cdot, 0)=0\},
\]
and
\[
 \mathscr{Y}=C^{\al, \frac{\al}{2\sigma}}(\Sn\times (0,T_{\va_0}]).
\]
Define $\mathcal{F}(v):=N|v|^{N-1}\frac{\pa v}{\pa t}+P_\sigma(v)$ for $v\in  \mathcal C^{2\sigma+\al, 1+\frac{\al}{2\sigma}}(\Sn\times (0,T_{\va_0}])$, and
\[
 L:\mathscr{X}\to \mathscr{Y},\quad \varphi \mapsto \mathcal{F}(w+\varphi)-\mathcal{F}(w).
\]
Note that $L(0)=0$,
\[
 L'(0)\varphi=Nw^{N-1}\varphi_t+P_\sigma(\varphi)+N(N-1)w^{N-2}w_t\varphi, \quad \forall \ \varphi\in \mathscr{X}.
\]
It follows from Proposition \ref{schsn12} that $L'(0): \mathscr{X}\to \mathscr{Y}$ is invertible when $\va_0$ is chosen sufficiently small.

By the Implicit Function Theorem, there exists a positive constant $\delta>0$ such that for any $\phi\in \mathscr{Y}$ with $\|\phi\|_{\mathscr{Y}}\leq \delta$ there exists a
unique solution $\varphi\in \mathscr{X}$ of the equation
\[
 L(\varphi)=\phi.
\]
Let $ T_*>0$ be small. Pick a cut off function $0\leq \eta(t)\leq 1$ in $\R_+$ satisfying $\eta(t)=1$ for $s\leq T_*$ and $\eta(t)=0$ if $s\geq 2T_*$. It is easy to see that
\[
 \|\eta(t)\mathcal{F}(w)\|_{\mathscr{Y}}\leq \delta
\]
provided $T_*$ is sufficiently small. Therefore, there exists a function $\varphi \in \mathscr{X}$ such that
\[
 L(\varphi)=-\eta(t)\mathcal{F}(w).
\]
Thus, $v:=w+\varphi$ satisfies $v(\cdot,0)=v_0$ and
\[
 \mathcal{F}(w+\varphi)=0,\quad \mbox{in } \Sn\times (0,T_*].
\]
Moreover, $v$ is positive if $T_*$ is small enough. The smoothness of $v$ follows from Corollary \ref{interior schsn} and bootstrap arguments.
\end{proof}

\subsection{Long time existence and convergence}

\begin{prop}\label{infinity norm controls all}
Let $v$ be a positive smooth solution of \eqref{Y2} in $\Sn\times (0,3]$ and satisfy $\Lambda^{-1}\leq v(y,t)\leq \Lambda$ for all $(y,t)\in \Sn\times (0,3]$ with some positive constant $\Lambda$. Then for any positive integer $k$, 
\be\label{higher order estimate}
\|v\|_{C^{k}(\Sn\times [2,3])}\leq C
\ee
where $C>0$ depends only on $n, \sigma, k$, $\Lambda$, and $r_{\sigma}^{g(1)}$.
\end{prop}

\begin{proof}
We first observe that $r_{\sigma}^{g(t)}$ is decreasing in $t$, and is lower bounded away from $0$ by Sobolev inequalities (see, e.g., \cite{Be}).
Hence through a scaling argument in $t$, we may assume that $v$ satisfies the equation $\frac{\pa v^N}{\pa t}=-P_{\sigma}(v)$ instead of \eqref{Y2}. 
By the H\"older estimates in \cite{AC} (see also Theorem 9.3 in \cite{dQRV}), there exists some $\beta\in(0,\min(1,2\sigma))$ such that
\[
|v|_{\beta,\frac{\beta}{2\sigma}; \Sn\times [1,3]}\leq C(n, \sigma, \beta, \Lambda).
\]
The Proposition follows from Corollary \ref{interior schsn} and bootstrap arguments.
\end{proof}

\begin{proof}[Proof of Theorem \ref{convergence of fractional yamabe flow on spheres}]
By Proposition \ref{STE}, we have a unique positive smooth solution of \eqref{Y1} on a maximum time interval $[0,T^*)$. Since the flow preserves the volume of the sphere, the Harnack inequality in Theorem \ref{thm:harnack} implies that $v(x,t)$ is uniformly bounded from above and away from zero. Proposition \ref{infinity norm controls all} yields smooth estimates for $v$ on $\Sn\times [\min(1,T^*/2),T^*)$. It follows that $T^*=\infty$, since otherwise by Proposition \ref{STE} we can extend $v$ beyond $T^*$.  Moreover, there exists $v_{\infty}\in C^{\infty}(\Sn)$ and a sequence $\{v(t_j)\}$ such that $v(t_j)$ converges smoothly to $v_{\infty}$. By Theorem \ref{convergence of flow 2} in the Appendix, $v(t)$ converges smoothly to $v_{\infty}$, i.e. there exists a smooth metric $g_{\infty}$ on $\Sn$ such that $g(t)$ converges smoothly to $g_{\infty}$. The formula for the gradient of the total $\sigma$-curvature gives
\[
\frac{d S}{d t}=-\frac{n-2\sigma}{2n}( vol_g(\Sn))^{\frac{2\sigma-n}{n}}\int_{\Sn}(R_{\sigma}^g-r_{\sigma}^g)^2 \ud vol_g.
\]
Thus,
\[
\int_0^{\infty}\int_{\Sn} (R_{\sigma}^g-r_{\sigma}^g)^2 \ud vol_g<\infty,
\]
which implies that $R_\sigma^{g_{\infty}}$ is a positive constant. 
\end{proof}

\section{Extinction profile of a fractional porous medium equation}\label{Extinction profile of a fractional porous medium equation}

Let $u(x,t)$ be the solution of \eqref{fpme} and $T>0$ be its extinction time. Since $u_0$ is not identically zero, it is proved in \cite{dQRV} that $u(x,t)>0$ in $\R^n\times (0,T)$ and $u(x,t)\in C^{\alpha}(\R^n\times (0,T))$ for some $\alpha \in (0,1)$.
We define $v(F(x),s)$ for all $x\in\R^n$ and all $s\geq 0$ as
\be\label{u to v}
v(F(x),s): = \left(\frac{1+\abs{x}^2}{2}\right)^{\frac{n-2\sigma}{2}}(T-t)^{-m/(1-m)} u(x,t)^{m}|_{t=T(1-e^{-s})},
\ee
where $F: \R^n\to\Sn$ is the inverse of stereographic projection from the north pole and $m=\frac{n-2\sigma}{n+2\sigma}$. By the assumption of $u_0$, we have $v(\cdot,0)\in C^2(\Sn)$. It follows from Proposition \ref{STE} that, there exists an $s^*>0$ and a unique positive function $\tilde v\in C^{\infty}(\Sn\times (0,s^*))$ satisfies 
\be\label{fpme for v}
\frac{\pa \tilde v^N}{\pa s}=-P_{\sigma}(\tilde v)+\frac{1}{1-m}\tilde v^N
\ee
and $\tilde v_0=v(\cdot,0)$. On the other hand, $\tilde u(x,t)$, which is defined by $\tilde v$ through \eqref{u to v}, satisfies \eqref{fpme}. By the uniqueness theorem on the solution of \eqref{fpme} in \cite{dQRV}, $v\equiv \tilde v$ in $\Sn\backslash\{\mathcal N\}\times (0,s^*)$, and hence $v$ can be extended to a positive and smooth function in $\Sn\times (0,s^*)$.

Our first goal is that $v$ defined by relation \eqref{u to v} is positive and smooth in $\Sn\times (0,\infty)$. Secondly, we will show that $v$ converges to a steady solution of \eqref{fpme for v}. In summary, we will show the following theorem in terms of $v$.

\begin{thm}\label{convergence of v}
Let $v$ be defined by relation \eqref{u to v}. Then $v$ is positive and smooth in $\Sn\times(0,\infty)$. Moreover, there is a unique positive solution $\bar v$ of
\be\label{steady equation of v}
-P_{\sigma}(\bar v)+\frac{1}{1-m}\bar v^N=0
\ee
such that
\[
\|v(y,s)-\bar v(y)\|_{C^3(\Sn)}\to 0\quad \text{as }s\to\infty.
\]
\end{thm}

Our proof of Theorem \ref{convergence of v} is inspired by some arguments in \cite{ds}. To prove convergence of $v(\cdot,t)$, we first establish the following universal estimates.
\begin{prop}\label{universal estimate}
Let $v$ be defined by relation \eqref{u to v}. There exist positive constants $\beta_1,\beta_2$ such that
\[
\beta_{1}\leq v(y,s)\leq \beta_2
\]
for all $y\in \Sn$, $s^*/2\le s< +\infty$. Hence, $v\in C^{\infty}(\Sn\times (s^*/2, \infty))$.
\end{prop}

\begin{proof}
Step1: We show that if $s_0$ is such that $v$ is positive and smooth in $\Sn\times (s^*/2,s_0)$, then there is a positive constant $\kappa_1$, independent of $s_0$, such that for all $s\in (s^*/2,s_0)$
\be\label{max lower bound}
\max_{\Sn} v(\cdot,s)>\kappa_1.
\ee
Let us argue by contradiction. If this is not true, then for every small $\va>0$, there is an $s_{\va}$ such that  $s_0>s_{\va}>s^*/2$ and $v(y,s_{\va})<\va$ for all $y\in \Sn$. Given $\va>0$, consider
\[
U(x,t)=K^{1/m}[(1+s_{\va}-\log T+\log (T-t))(T-t)]_+^{\frac{1}{1-m}}\left(\frac{2}{1+\abs x^2}\right)^{\frac{n+2\sigma}{2}},
\]
where $K$ will be chosen later. Direct computations yield that
\[
\begin{split}
&U_t-(\Delta)^{\sigma} U^{\frac{n-2\sigma}{n+2\sigma}}\\
&=\quad K^{\frac{1}{m}}[(1+s_{\va}-\log T+\log (T-t))(T-t)]_+^{\frac{m}{1-m}}\left(\frac{2}{1+\abs x^2}\right)^{\frac{n+2\sigma}{2}}\\
&\quad\quad \cdot  \left(\log T-\log(T-t)-2-s_{\va}+P_{\sigma}(1)K^{1-1/m}\right),
\end{split}
\]
where we used that $(-\Delta)^{\sigma}\left(\frac{2}{1+\abs x^2}\right)^{\frac{n-2\sigma}{2}}=P_{\sigma}(1)\left(\frac{2}{1+\abs x^2}\right)^{\frac{n+2\sigma}{2}}$ with $P_{\sigma}(1)$ given in \eqref{value on spherical harmonics}.

Let $t_{\va}$ be that $s_{\va}-\log T+\log (T-t_{\va})=0$. We choose $K$ small such that $P_{\sigma}(1)K^{\frac{m-1}{m}}>2$ and let $\va=K$. Since $v(y,s_{\va})<\va$,
\[
u(x,t_{\va})<{\va}^{1/m}(T-t_{\va})^{\frac{1}{1-m}}\left(\frac{2}{1+\abs x^2}\right)^{\frac{n+2\sigma}{2}}=U(x,t_{\va}).
\]
For $t>t_{\va}$, $U(x,t)$ is a supersolution of \eqref{fpme}. It follows from the comparison principle (see the proof of Theorem 6.2 in \cite{dQRV}) that $u(x,t)\leq U(x,t)$. But $U$ vanishes before $T$. Hence, $u$ vanishes before $T$, which contradicts the definition of the extinction time $T$. 

\medskip

Step 2: $v$ is strictly positive and smooth for $s^*/2<s<\infty$.

To show this, we define
\[
s_0=\sup\{s>0: v\in C^{3,1}(\Sn\times(s^*/2,s))\}.
\]
Note that $s_0\geq s^*$. We assume that $s_0<\infty$. Since $v\in C^{3,1}(\Sn\times(s^*/2,s_0))$ and $v$ is positive, by Theorem \ref{thm:harnack} and step 1 we have that $v$ is uniformly lower bounded away from 0. We define
\[
U(x,t)=(M-t)_+^{1/(1-m)} k(n,\sigma)\left(\frac{1}{1+\abs x^2}\right)^{\frac{n+2\sigma}{2}},
\]
where $k(n,\sigma)$ is defined in Theorem \ref{extinction profile}. $U(x,t)$ satisfies \eqref{fpme} and will be used as a barrier function.
By our assumptions on $u_0$,  we choose sufficiently large $M>T$ such that
\[
u_0(x)\leq M^{1/(1-m)} k(n,\sigma)\left(\frac{1}{1+\abs x^2}\right)^{\frac{n+2\sigma}{2}}.
\]
 It follows from comparison principle (Theorem 6.2 in \cite{dQRV}) that for all $0<t<T$,
\[
u(x,t)\leq (M-t)^{1/(1-m)} k(n,\sigma)\left(\frac{1}{1+\abs x^2}\right)^{\frac{n+2\sigma}{2}}.
\]
Hence, for all $s^*/2\leq s\leq s_0$
\[
v(y,s)\leq \left(\frac{T+(M-T)e^s}{T}\right)^{\frac{m}{1-m}}k(n,\sigma)^m\leq\left(\frac{T+(M-T)e^{s_0}}{T}\right)^{\frac{m}{1-m}}k(n,\sigma)^m.
\]
It follows that $v$ is uniformly bounded from above. Since $v$ satisfies \eqref{fpme for v}, Proposition \ref{infinity norm controls all} implies that $v$ has a uniform limit as $s\to s_0$ which is also positive and smooth. By Proposition \ref{STE} $v$ can be extended in a smooth and positive way beyond $s_0$, which violates the definition of $s_0$. We conclude that $s_0=+\infty$.

\medskip

Step 3: There is a constant $\kappa_2=(1+P_{\sigma}(1)(1-m))^{m/(1-m)}>0$ such that for all $s>0$
\[
\min_{\Sn}v(y,s)\leq\kappa_2.
\]
We argue by contradiction. Suppose that there is a time $\bar s<\infty$ for which
\[
\min_{\Sn}v(y,\bar s)>\kappa_2.
\]
This implies
\[
u(x,\bar t)\geq (T-\bar t+P_{\sigma}(1)(1-m)(T-\bar t))^{1/(1-m)}\left(\frac{1}{1+\abs x^2}\right)^{\frac{n+2\sigma}{2}},
\]
where $\bar t=T(1-e^{-\bar s})<T$.
We consider a barrier function
\[
U(x,t)=\left(T-\bar t+P_{\sigma}(1)(1-m)(T-t)\right)^{\frac{1}{1-m}}\left(\frac{1}{1+\abs x^2}\right)^{\frac{n+2\sigma}{2}},
\]
which satisfies \eqref{fpme}. Since $u(x,\bar t)\geq U(x,\bar t)$, by the comparison principle
\[
u(x,\bar t)\geq \left(T-\bar t+P_{\sigma}(1)(1-m)(T-t)\right)^{\frac{1}{1-m}}\left(\frac{1}{1+\abs x^2}\right)^{\frac{n+2\sigma}{2}}.
\]
This contradicts the extinction time $T$ of $u$.

From Steps 1, 2 and 3 we can conclude Proposition \ref{universal estimate} by taking $\beta_2=C\kappa_2$ and $\beta_1=\kappa_1/C$ where $C$ is the constant in Theorem \ref{thm:harnack} for $s_0=\infty$.
\end{proof}

Now we are in the position to prove Theorem \ref{convergence of v}. Let $J$ be the functional defined as
\[
J(z)=\frac{1}{2}\int_{\Sn}zP_{\sigma}(z)-\frac{1}{(1-m)(N+1)}\int_{\Sn} z^{N+1}.
\]
Direct computations yield
\begin{lem}\label{decreasing energy}
Let $v(x,s)$ satisfy \eqref{fpme for v}. Then
\[
\frac{\ud}{\ud s}J(v(\cdot,s))=-N\int_{\Sn}v^{N-1}(v_s)^2\leq 0.
\]
\end{lem}

The above lemma indicates that the functional is decreasing in time. The next lemma states that this functional is always nonnegative, and hence $\lim\limits_{s\to\infty}J(v(\cdot,s))$ exists.
\begin{lem}\label{positive energy}
$J(v(\cdot,s))\geq 0$ for all $s>0$.
\end{lem}

\begin{proof}
The proof is similar to that of Lemma 6.1 in \cite{ds}, which is included here for completeness. We argue by contradiction. Assume that for certain $0<s_0<\infty$ one has $J(v(\cdot,s_0))< 0$. By Lemma \ref{decreasing energy} $J(v(\cdot,s))< 0$ for all $s>s_0$. Let us consider the quantity
\[
F(s)=\int_{\Sn}v^{N+1}(y,s) \ud y\geq 0,\quad s\in(0,\infty).
\]
Then
\[
\begin{split}
\frac{N}{N+1}\frac{\ud}{\ud s}F(s)
=\int_{\Sn}(v^N)_s v
&=-2J(v(\cdot,s))+\frac{N-1}{(1-m)(N+1)}F(s)\\
&\geq \frac{N-1}{(1-m)(N+1)}F(s)
\end{split}
\]
for all $s>s_0$. Note that $F(s)\neq 0$ for all $s\geq s_0$. Otherwise, $v(\cdot,s)\equiv 0$ which is impossible because $J(v(\cdot,s))\leq J(v(\cdot,s_0))< 0$. Integrating the above differential inequality, we have
\[
F(s)\geq F(s_0)e^{s-s_0}.
\]
 It follows that $F(s)\to \infty$ as $s\to \infty$. On the other hand, Proposition \ref{universal estimate} implies that $v$ is uniformly bounded. Consequently, $F(s)$ is bounded. We reach a contradiction.
\end{proof}

\begin{proof}[Proof of Theorem \ref{convergence of v}]
It follows from Proposition \ref{universal estimate} 
and Proposition \ref{infinity norm controls all} 
that for $s>s^*/2$, $v(\cdot,s)$ is compact in $C^k(\Sn)$ for any $k$. Let $\bar v$ be a limit point of $v(\cdot,s)$ as $s\to\infty$ in the $C^2$ sense. We will show that $\bar v$ is a solution of \eqref{steady equation of v} and $\bar v$ is the unique limit of $v(\cdot,s)$ as $s\to\infty$.

Suppose that along a sequence $s_j\to\infty$, $v(\cdot,s_j)\to\bar v$ in $C^2(\Sn)$. Since
\[
\frac{\ud}{\ud s}J(v(\cdot,s))=-N\int_{\Sn}v^{N-1}v_s^2=-\frac{4N}{(N+1)^2}\int_{\Sn}\abs{(v^{(N+1)/2}(\cdot,s))_s},
\]
we have, by integrating from $s_j$ to $s_j+\tau$ and using the Cauchy-Schwarz inequality,
\[
\begin{split}
&\int_{\Sn}\abs{v^{\frac{N+1}{2}}(\cdot,s_j+\tau)-v^{\frac{N+1}{2}}(\cdot,s_j)}^2\\
&\quad\leq \frac{(N+1)^2\tau}{4N}\big(J(v(\cdot,s_j))-J(v(\cdot,s_j+\tau))\big).
\end{split}
\]
By Lemma \ref{decreasing energy} and Lemma \ref{positive energy}, $J(v(\cdot,s))$ has a limit as $s\to\infty$. Hence for each $\tau>0$, $\{v(\cdot,s_j+\tau)\}_{1}^{\infty}$ is Cauchy in $L^{N+1}$. It follows that $v(\cdot,s_j+\tau)\to\bar v$ in $L^{N+1}$, and in $C^2(\Sn)$ uniformly in $\tau$ for $\tau$ in bounded intervals. Thus, for any $\phi\in C^{\infty}(\Sn)$ we have, 
\[
\begin{split}
&\int_{\Sn}\big(v^N(\cdot,s_j+1)-v^N(\cdot,s_j)\big)\phi\\
&\quad=\int_0^1\int_{\Sn}\left(-P_{\sigma}\big(v(y,s_j+\tau)\big)+\frac{1}{1-m}v^N(y,s_j+\tau)\right)\phi\ud y\ud\tau .
\end{split}
\]
After sending $j\to\infty$, we obtain
\[
\int_{\Sn}\left(-P_{\sigma}(\bar v)+\frac{1}{1-m}\bar v^N\right)\phi=0,
\]
i.e.,  $\bar v$ solves \eqref{steady equation of v}. Finally, it follows from Theorem \ref{convergence of flow} that $v(\cdot,s)$ converges to $\bar v$ in $C^3(\Sn)$.
\end{proof}
\begin{proof}[Proof of Theorem \ref{extinction profile}]
By the classification of solutions of \eqref{steady equation of v} in \cite{CLO} and \cite{Li04}, Theorem \ref{extinction profile} follows from Theorem \ref{convergence of v} immediately.
\end{proof}

From Theorem \ref{extinction profile} we see that the extinction profile of $u(x,t)$ is determined by the pair of numbers $(\lda, x_0)=(\lda (u_0), x_0(u_0))$. The next theorem verifies the stability of both the extinction time and the extinction profile.

\begin{thm}\label{stability}
$T(u_0),\lda (u_0)$ and $x_0(u_0)$ continuously depend on $u_0$ in the sense that if
$u_0$, $\{u_{0;j}\}$ are positive $C^2$ functions in $\R^n$, $(u_0^m)_{0,1}, (u_{0;j}^m)_{0,1}$ can be extended to positive $C^2$ functions near the origin, and $\lim\limits_{j\to\infty} \|u^m_{0;j}-u^m_0\|_{b}=0$ where $\|\cdot\|_b$ is defined by
\[
\|\cdot\|_b=\|\cdot\|_{C^2(B_2)}+\|(\cdot)_{0,1}\|_{C^2(B_2)},
\]
then
\[
\lim_{j\to\infty}(T(u_{0;j}),\lda (u_{0;j}), x_0(u_{0;j}))=(T(u_0),\lda (u_0), x_0(u_0)).
\]
\end{thm}
\begin{proof}
Given Theorem \ref{l}, Lemma \ref{d/dJ} and Theorem \ref{convergence of flow}, the proof is identical to the proof of Theorem 1.2 in \cite{ds}. We refer to \cite{ds} for details.
\end{proof}

\section{A Sobolev inequality and a Hardy-Littlewood-Sobolev inequality along a fractional diffusion equation}
\label{inequality along flow}

As mentioned in the introduction, the results in this section are inspired by \cite{CCL} and \cite{D}.
\begin{prop}\label{increase H}
Assume that $n\geq 2$. If $u$ is a solution of \eqref{fpme} with positive initial data $u_0\in C^2$ in $\R^n$ satisfying that $(u_0^m)_{0,1}$ can be extended to a positive $C^2$ function near the origin, then
\[
\frac{1}{2}\frac{\ud}{\ud t} H =\left(\int_{\R^n}u^{m+1}\right)^{\frac{2\sigma}{n}}\left(S_{n,\sigma}\|u^m\|^2_{\dot H^{\sigma}}-\|u^m\|^2_{L^{2^*(\sigma)}}\right)\geq 0,
\]
where $H$ is given by \eqref{definition of H}.
\end{prop}

\begin{proof}
It follows from \eqref{fpme} and \eqref{eq:fs} that
\[
\begin{split}
\frac{\ud}{\ud t} H
&=\int_{\R^n}2u(-\Delta)^{-\sigma}u_t\ud x-2S_{n,\sigma}\left(\int_{\R^n}u^{m+1}\right)^{\frac{2\sigma}{n}}\int_{\R^n}u^{m}u_t\\
%&=\int_{\R^n}-2u(-\Delta)^{-\sigma}(-\Delta)^{\sigma}u^m\ud x+2S_{n,\sigma}\left(\int_{\R^n}u^{m+1}\right)^{\frac{2\sigma}{n}}\int_{\R^n}u^{m}(-\Delta)^{\sigma}u^m\\
&= -2\int_{\R^n}u^{m+1}+2S_{n,\sigma}\left(\int_{\R^n}u^{m+1}\right)^{\frac{2\sigma}{n}}\int_{\R^n}u^{m}(-\Delta)^{\sigma}u^m \\
&=2\left(\int_{\R^n}u^{m+1}\right)^{\frac{2\sigma}{n}}\left(S_{n,\sigma}\|u^m\|^2_{\dot H^{\sigma}}-\|u^m\|^2_{L^{2^*(\sigma)}}\right)\geq 0.
\end{split}
\]
%In the last equality we have used Plancherel Formula since $(-\Delta)^{\sigma}u^m=-u_t\in L^2(\R^n)$ which leads to $\int_{\R^n}u^{m}(-\Delta)^{\sigma}u^m=\|\abs{\xi}^{2\sigma} \mathcal F u^m\|^2_{L^2(\R^n)}=\|u^m\|^2_{\dot H^{\sigma}}$ and $\mathcal F$ denotes the Fourier transform.
Note that the first part of Theorem \ref{convergence of v}, i.e., $v$ defined by \eqref{u to v} is positive and smooth in $\Sn\times(0,\infty)$, has been used in the justifications of these equalities.
\end{proof}

The next lemma gives an estimate for the extinction time of solutions of \eqref{fpme}.
\begin{lem}\label{estimate of extinction time}
If $u$ is a solution of \eqref{fpme} with positive initial data $u_0\in C^2$ in $\R^n$ satisfying that $(u_0^m)_{0,1}$ can be extended to a positive $C^2$ function near the origin, then for any $t\in(0,T)$ we have
\[
\left(\frac{4\sigma(T-t)}{(n+2\sigma)S_{n,\sigma}}\right)^{\frac{n}{2\sigma}}\leq \int_{\R^n} u^{m+1}(t,x) dx\leq
\int_{\R^n}u_0^{m+1} dx.
\]
Consequently, the extinction time $T$ is bounded by
\[
T\leq\frac{(n+2\sigma)S_{n,\sigma}}{4\sigma}\left(\int_{\R^n}u_0 ^{m+1}dx\right)^{\frac{2\sigma}{n}}.
\]
If in addition $n>4\sigma$, then
\[
T\geq \frac{(n+2\sigma)}{2n}\frac{\int_{\R^n}u_0 ^{m+1}dx}{\int_{\R^n}u_0^m(-\Delta)^{\sigma}u_0^m}
\]
and
\[
\begin{split}
\int_{\R^n}u^m(\cdot,t)(-\Delta)^{\sigma}u^m(\cdot,t)&\leq \int_{\R^n}u_0^m(-\Delta)^{\sigma}u_0^m,\\
\int_{\R^n}u^{m+1}(\cdot,t)&\geq\int_{\R^n}u_0^{m+1}-\frac{2n}{n+2\sigma}t\int_{\R^n}u_0^m(-\Delta)^{\sigma}u_0^m.
\end{split}
\]
\end{lem}

\begin{proof}
As in the proof of Lemma \ref{positive energy}, we define
\be\label{definition of F}
F(t):=\int_{\R^n}u^{m+1}(x,t)\ud x,
\ee
which is positive in $(0,T)$ and $F(T)=0$. It follows that
\[
\begin{split}
F'(t)&=(m+1)\int_{\R^n}u^m(\cdot,t)u_t(\cdot,t)\\
&=-(m+1)\int_{\R^n}u^m(\cdot,t)(-\Delta)^{\sigma}u^m(\cdot,t)\leq -\frac{m+1}{S_{n,\sigma}} F(t)^{1-\frac{2\sigma}{n}},
\end{split}
\]
where we have used the Sobolev inequality \eqref{eq:fs} in the last inequality.
This shows the first two inequalities by simple integrations.
If in addition $n>4\sigma$, then
\[
\begin{split}
F''(t)&=m(m+1)\int_{\R^n}u^{m-1}(\cdot,t)\big((-\Delta)^{\sigma}u^m(\cdot,t)\big)^2\\
&\quad +m(m+1)\int_{\R^n}u^{m}(\cdot,t)(-\Delta^{\sigma})\big(u^{m-1}(-\Delta)^{\sigma}u^m(\cdot,t)\big)\\
&=2m(m+1)\int_{\R^n}u^{m-1}(\cdot,t)\big((-\Delta)^{\sigma}u^m(\cdot,t)\big)^2\geq 0,
\end{split}
\]
where the condition $n>4\sigma$ is used to guarantee the $L^2$ integrability of $u^m(\cdot,t)$ such that we can use Plancherel's theorem in the second equality.
Thus, the lower bound of $T$ follows from that $0=F(T)\geq F(t)+F'(t)(T-t)$ with sending $t\to 0$. The last two inequalities follows from the sign of $F''$ and simple integrations.
\end{proof}

Let
\be\label{definition of QEG}
Q:=-\frac{1}{m+1}F'F^{\frac{2\sigma-n}{n}}, E:=-\frac{1}{m+1}F'F^{-1}, G(t_1,t_2):=\exp\left((m+1)\int_{t_1}^{t_2}E(s)\ud s\right).
\ee
\begin{thm} Assume $n>4\sigma$. For any $u_0$ positive and  $C^2$ in $\R^n$ satisfying that $(u_0^m)_{0,1}$ can be extended to a positive $C^2$ function near the origin, we have
\[
\begin{split}
S_{n,\sigma}\|u_0\|^2_{L^{\frac{2n}{n+2\sigma}}}&-\int_{\R^n}u_0(-\Delta)^{-\sigma}u_0\ud x +4mS_{n,\sigma}\int_0^T\ud t\int_0^t F(s)^{\frac{2\sigma}{n}}K(s)G(t,s)\ud s\\
&=2\|u_0^m\|_{L^{2^*(\sigma)}}^{\frac{4\sigma}{n-2\sigma}}\left(S_{n,\sigma}\|u_0^m\|^2_{\dot H^{\sigma}}-\|u_0^m\|^2_{L^{2^*(\sigma)}}\right)\int_0^T G(t,0)\ud t,
\end{split}
\]
where $u(\cdot,t)$ is the solution of \eqref{fpme} with initial data $u(\cdot,t)=u_0$, $T$ is the extinction time of $u(\cdot,t)$ and $F, E, G, K$ are defined in \eqref{definition of F}, \eqref{definition of QEG} and \eqref{definition of K}.
\end{thm}
\begin{proof}
From the proof of Proposition \ref{increase H} we know that
\[
H'(t)=2F(t)\big(S_{n,\sigma}Q(t)-1\big).
\]
Hence
\[
\begin{split}
H''(t)&=2F'(t)\big(S_{n,\sigma}Q(t)-1\big)+2F(t)S_{n,\sigma}Q'(t)\\
&=\frac{F'(t)}{F(t)}H'(t)+2F(t)S_{n,\sigma}Q'(t)\\
&=-(m+1)E(t)H'(t)+2F(t)S_{n,\sigma}Q'(t).
\end{split}
\]
On the other hand,% \textbf(since $F''$ can be justified)
\[
\begin{split}
Q'(t)&=\frac{F''(t)-\frac{n-2\sigma}{n}F^{-1}(t)\big(F'(t)\big)^2}{-(m+1)F(t)^{\frac{n-2\sigma}{n}}}\\
&=-\frac{2m}{F(t)^{\frac{n-2\sigma}{n}}}\left(\int_{\R^n}u^{m-1}(\cdot,t)\big((-\Delta)^{\sigma}u^m(\cdot,t)\big)^2 -F^{-1}\int_{\R^n}u^m(\cdot,t)(-\Delta)^{\sigma}u^m(\cdot,t)\right)\\
&=-\frac{2m}{F(t)^{\frac{n-2\sigma}{n}}}\int_{\R^n} u(\cdot,t)^{m-1} \abs{-(-\Delta)^{\sigma} u(\cdot,t)^m+E(t) u(\cdot,t)}^2.
\end{split}
\]
Denote
\be\label{definition of K}
K(t):=\int_{\R^n} u(\cdot,t)^{m-1} \abs{-(-\Delta)^{\sigma} u(\cdot,t)^m+E(t) u(\cdot,t)}^2.
\ee
Then
\[
H''(t)=-(m+1)E(t)H'(t)-4mF^{\frac{2\sigma}{n}}(t)S_{n,\sigma}K(t).
\]
Multiplying $G(0,s)$ and integrating from $0$ to $t$, we have
\[
H'(t)G(0,t)-H'(0)G(0,0)=\int_0^t(H'G)'(s)\ud s=-4mS_{n,\sigma}\int_0^t F(s)^{\frac{2\sigma}{n}}K(s)G(0,s)\ud s.
\]
Dividing $G(0,t)$ and integrating from $0$ to $T$, we obtain
\[
0-H(0)=H'(0)\int_0^T G(t,0)\ud t-4mS_{n,\sigma}\int_0^T\ud t\int_0^t F(s)^{\frac{2\sigma}{n}}K(s)G(t,s)\ud s,
\]
which finishes the proof.
\end{proof}

The drawback of the above Theorem is that the extra terms are not explicit. Fortunately, we can use simple estimates to reach Theorem \ref{duality}.

\begin{proof}[Proof of Theorem \ref{duality}]
We first assume that $w=u^m_0$ where $u_0\in C^2(\R^n)$ is positive and satisfies that $(u_0^m)_{0,1}$ can be extended to a positive $C^2$ function near the origin. 
By Lemma \ref{estimate of extinction time},
\[
(m+1)E(s)\geq (m+1)S^{-1}_{n,\sigma}\left(\int_{\R^n}u(\cdot,s)^{m+1}\right)^{-2\sigma/n} \geq (m+1)S^{-1}_{n,\sigma}\left(\int_{\R^n}u_0^{m+1}\right)^{-2\sigma/n}=:b.
\]
By Lemma \ref{estimate of extinction time} again, we have $bT\leq \frac{n}{2\sigma}$. Therefore,
\[
\begin{split}
\int_0^T G(t,0)\ud t \leq \int_0^T e^{-bt}\ud t&=\frac{1-e^{-bT}}{b}\\
&\leq\frac{1-e^{-\frac{n}{2\sigma}}}{m+1}S_{n,\sigma}\left(\int_{\R^n}u_0^{m+1}\right)^{2\sigma/n}.
\end{split}
\]
Hence \eqref{eq:estimate for the remainder term} holds for $w=u^m_0$ where $u_0\in C^2(\R^n)$ is positive and satisfies that $(u_0^m)_{0,1}$ can be extended to a positive $C^2$ function near the origin.  

For any nonnegative $u\in C_c^{\infty}(\R^n)$, we consider $w_{\va}=u+\va (\frac{2}{1+|x|^2})^{\frac{n-2\sigma}{2}}$ with $\va>0$. Then \eqref{eq:estimate for the remainder term} holds for $w_{\va}$. 
By sending $\va\to 0$, we have \eqref{eq:estimate for the remainder term} for $u$. Finally, Theorem \ref{duality} follows from a density argument.
\end{proof}

\appendix

\section{Uniqueness theorem for negative gradient flow involving nonlocal operator}\label{appendix a}
In this appendix, we provide a uniqueness theorem for fractional Yamabe flows, which is analog to L. Simon's uniqueness Theorem in \cite{S}. The proofs are essentially the same and we will just sketch them in our setting. Denote $H^{\sigma}(\Sn)$ as the closure of $C^{\infty}(\Sn)$ under the norm
\[
\|v\|_{H^{\sigma}(\Sn)}=\int_{\Sn}vP_{\sigma}(v).
\]
Let $\al\in (0,1)$ such that $2\sigma+\al$ is not an integer. Let $J$ be the functional defined as
\[
J(v)=\frac{1}{2}\int_{\Sn}vP_{\sigma}(v)-\frac{1}{(1-m)(N+1)}\int_{\Sn} v^{N+1},\quad v\in H^{\sigma}(\Sn).
\]
Then
\[
\nabla J (v)= P_{\sigma}(v)-\frac{1}{1-m}v^N.
\]
Let $\bar v$ be such that $\nabla J (\bar v)=0$.
\begin{thm}\label{l}
There exist $\theta\in (0,1/2)$ and $r_0>0$ such that for any $v\in C^{2\sigma+\al}(\Sn)$ with $\|v-\bar v\|_{C^{2\sigma+\al}}<r_0$,
\[
\|\nabla J (v) \|_{L^2(\Sn)}\geq \abs{J(v)-J(\bar v)}^{1-\theta}.
\]
\end{thm}
\begin{proof}
Since we have Schauder estimates (see, e.g., \cite{JLX1}) and $L^2$ estimates (which is free from equivalence of definitions of fractional Sobolev spaces on $\Sn$ ) for $P_{\sigma}$, the proof is identical to that of Theorem 3 in \cite{S}.
\end{proof}

Let $v(x,s)$ and $\bar v$ be as in Section \ref{Extinction profile of a fractional porous medium equation}. Then direction computations and uniform bounds of $v(x,s)$ yield the following lemma
\begin{lem}\label{d/dJ}
There exist two constants $c_0$ and $T_0$ such that for any $t>T_0$ we have
\[
-\frac{\ud}{\ud s} J(v(\cdot,s))\geq c_0 \|v_s\|_{L^2(\Sn)}\|\nabla J(v(\cdot,s))\|_{L^2(\Sn)}.
\]
\end{lem}

\begin{thm}\label{convergence of flow}
\[
\lim_{s\to\infty}\|v(\cdot,s)-\bar v\|_{C^l(\Sn)}=0
\]
for any positive integer $l$.
\end{thm}
\begin{proof}
First we can prove that $v(\cdot,t)$ converges to $\bar v$ in $C^{2\sigma+\al}(\Sn)$, using the same methods as the proof of Proposition 21 in \cite{A} or the proof of Theorem 1 in \cite{GW}. Then Theorem \ref{convergence of flow} follows from the uniform $C^{l+1}$ bound of $v(x,s)$.
\end{proof}

Similarly if
\[
S(z)=\frac{\int_{\Sn}zP_{\sigma}(z)}{\left(\int_{\Sn} z^{N+1}\right)^{\frac{2}{N+1}}},\quad z\in H^{\sigma}(\Sn),
\]
then
\[
\nabla S (z)=2\left(\int_{\Sn} z^{N+1}\right)^{-\frac{2}{N+1}}\left(P_{\sigma}(z)-\frac{\int_{\Sn}zP_{\sigma}(z)}{\int_{\Sn} z^{N+1}}z^N\right).
\]

Let $v(x,t)$ and $v_{\infty}$ be as in Theorem \ref{convergence of fractional yamabe flow on spheres}. Note that $\nabla S (v_{\infty})=0$.
\begin{thm}\label{l_2}
There exist $\theta\in (0,1/2)$ and $r_0>0$ such that for any $\|v-v_{\infty}\|_{C^{2\sigma+\al}}<r_0$,
\[
\|\nabla S (v) \|_{L^2(\Sn)}\geq \abs{S(v)-S(v_{\infty})}^{1-\theta}.
\]
\end{thm}
\begin{lem}\label{d/dS}
There exist two constants $c_0$ and $T_0$ such that for any $t>T_0$ we have
\[
-\frac{\ud}{\ud t} S(v(\cdot,t))\geq c_0 \|v_t\|_{L^2(\Sn)}\|\nabla S(v(\cdot,t))\|_{L^2(\Sn)}.
\]
\end{lem}

\begin{thm}\label{convergence of flow 2}
\[
\lim_{t\to\infty}\|v(\cdot,t)-v_{\infty}\|_{C^l(\Sn)}=0
\]
for any positive integer $l$.
\end{thm}

\section{Some interpolation inequalities}\label{appendix b}

\begin{lem}\label{interpolation inequality}
Suppose that $0<\al<\min(1,2\sigma)$ and $2\sigma+\al$ is not an integer. There exists a constant $C>0$ depending only on $n$ and $\sigma$ such that for any $\va>0$ and $u\in \mathcal C^{2\sigma+\al,1+\frac{\al}{2\sigma}}(Q_T)$, we have
\begin{align}
\label{interpolation:1}
|u_t|_{0;Q_T}&\leq \va [u_t]_{\al,\frac{\al}{2\sigma};Q_T}+C\va^{-2\sigma/\al} |u|_{0;Q_T},\\ 
\label{interpolation:2}
|(-\Delta)^{\sigma}u|_{0;Q_T}&\leq \va [u]_{2\sigma+\al,1+\frac{\al}{2\sigma};Q_T}+C\va^{-2\sigma/\al} |u|_{0;Q_T},\\ 
\label{interpolation:3}
[u]_{\al,\frac{\al}{2\sigma};Q_T}&\leq \va [u]_{2\sigma+\al,1+\frac{\al}{2\sigma};Q_T}+C\va^{-\al/(2\sigma)} |u|_{0;Q_T}.
\end{align}
If $\sigma>\frac{1}{2}$, then
\be\label{interpolation:4}
[\nabla_x u]_{\al,\frac{\al}{2\sigma};Q_T}\leq \va [u]_{2\sigma+\al,1+\frac{\al}{2\sigma};Q_T}+C\va^{-(1+\al)/(2\sigma-1)} |u|_{0;Q_T}.
\ee
\end{lem}

\begin{proof}
By the fractional parabolic dilations of the form $u(x,t)\to u(Rx,R^{2\sigma}t)$, we only need to show the case $\va=1$ and $T=2$.
Take $X=(x,t)\in Q_T$ and we have, for some $\theta\in (-1,1)$,
\[
\begin{split}
|u_t(X)|&\leq |u_t(X)-\big(u(x,t\pm 1)-u(x,t)\big)|+2|u|_{0;Q_T}\\
&=|u_t(X)-u_t(x,t+\theta)|+2|u|_{0;Q_T}\leq  [u_t]_{\al,\frac{\al}{2\sigma};Q_T}+2|u|_{0;Q_T}.
\end{split}
\] 
This finishes the proof of \eqref{interpolation:1}. For \eqref{interpolation:2} and \eqref{interpolation:3}, we first recall (see, e.g., \cite{Stein}) that
\[
|w|_{2\sigma+\al;\R^n}\leq C(|w|_{0;\R^n}+|(-\Delta)^{\sigma}w|_{\al;\R^n})\quad\text{for all }w\in C^{2\sigma+\al}(\R^n).
\]
Hence,
\[
\begin{split}
|(-\Delta)^{\sigma}u(x,t)|&\leq C(|u(\cdot, t)|_{0;\R^n}+|u(\cdot,t)|_{C^{2\sigma+\al}(\R^n)})\\
&\leq C(|u|_{0;Q_T}+[(-\Delta)^{\sigma}u]_{\al,\frac{\al}{2\sigma};Q_T})\leq C([u]_{2\sigma+\al,1+\frac{\al}{2\sigma};Q_T}+ |u|_{0;Q_T}),
\end{split}
\]
and
\[
\begin{split}
[u]_{\al,\frac{\al}{2\sigma};Q_T}&\leq\sup_{t_1\neq t_2, x}\frac{|u(x,t_1)-u(x,t_2)|}{|t_1-t_2|^{\frac{\al}{2\sigma}}}+\sup_{x_1\neq x_2, t}\frac{|u(x_1,t)-u(x_2,t)|}{|x_1-x_2|^{\al}}\\
&\leq C( |u|_{0;Q_T}+|u_t|_{0;Q_T}+\sup_t |u(\cdot,t)|_{2\sigma+\al;\R^n})\\
&\leq C([u]_{2\sigma+\al,1+\frac{\al}{2\sigma};Q_T}+|u|_{0;Q_T}).
\end{split}
\]
Finally, for $\sigma>\frac{1}{2}$ we notice that by the same methods as above,
\[
\sup_{t,x_1\neq x_2}\frac{|\nabla_x u(x_1,t)-\nabla_x u(x_2,t)|}{|x_1-x_2|^{\al}}\leq C([u]_{2\sigma+\al,1+\frac{\al}{2\sigma};Q_T}+|u|_{0;Q_T}).
\]
Thus, to prove \eqref{interpolation:4}, we only need to show
\[
\sup_{s\neq t, x}\frac{|\nabla_x u(x,s)-\nabla_x u(x,t)|}{|s-t|^{\frac{\al}{2\sigma}}}\leq C([u]_{2\sigma+\al,1+\frac{\al}{2\sigma};Q_T}+|u|_{0;Q_T}).
\]
Fix any $x_0\in \R^n$. Let $w(x,t)=(-\Delta)^{\sigma}u(x,t)$ and $\eta(x)$ be a smooth cut-off function supported in $B_2(x_0)\in\R^n$ and equal to $1$ in $B_1(x_0)$. Let
\[
u_0(x,t)=(-\Delta)^{-\sigma}(\eta w)=\int_{\R^n}\frac{\eta(y)w(y,t)}{|x-y|^{n-2\sigma}}\,\ud y.
\]
For convenience we have omitted some positive constant as in \eqref{sigma potential}. Then
\[
(-\Delta)^{\sigma}(u_0(x,t)-u(x,t)-u_0(x,s)+u(x,s))=0 \quad\mbox{in } B_1(x_0),
\]
which implies, for $0<|t-s|\leq 1$,
\[
\begin{split}
&|\nabla_x u_0(x_0,t)-\nabla_x u(x_0,t)-\nabla_x u_0(x_0,s)+\nabla_x u(x_0,s)|\\
&\leq C|u_0(x,t)-u(x,t)-u_0(x,s)+u(x,s)|_{L^{\infty}(\R^n)}\\
&\leq C(|u_t|_{0;Q_T}+[u]_{2\sigma+\al,1+\frac{\al}{2\sigma};Q_T})|t-s|^{\frac{\al}{2\sigma}}.
\end{split}
\]
Since $\sigma>1/2$ and 
\[
\nabla_x u_0(x_0,t)=(2\sigma-n)\int_{\R^n}\frac{(x_0-y)\eta(y)w(y,t)}{|x_0-y|^{n+2-2\sigma}}\,\ud y,
\]
we have
\[
|\nabla_x u_0(x_0,t)-\nabla_x u_0(x_0,s)|\leq C[u]_{2\sigma+\al,1+\frac{\al}{2\sigma};Q_T}|t-s|^{\frac{\al}{2\sigma}}.
\]
Together with \eqref{interpolation:1}, we have
\[
\sup_{s\neq t, x}\frac{|\nabla_x u(x,s)-\nabla_x u(x,t)|}{|s-t|^{\frac{\al}{2\sigma}}}\leq C([u]_{2\sigma+\al,1+\frac{\al}{2\sigma};Q_T}+|u|_{0;Q_T}).
\]
This finishes the proof of \eqref{interpolation:4}.
\end{proof}

\begin{lem}\label{interpolation inequality on Sn}
Suppose that $0<\al<\min(1,2\sigma)$ and $2\sigma+\al$ is not an integer. For any small $\va>0$, there exists a constant $C(\va)>0$ depending only on $n,\sigma$ and $\va$ such that for any $v\in \mathcal C^{2\sigma+\al,1+\frac{\al}{2\sigma}}(\mathcal Q_T)$, we have
\begin{align}
\label{interpolation:1s}
|v_t|_{0;\mathcal Q_T}&\leq \va [v_t]_{\al,\frac{\al}{2\sigma};\mathcal Q_T}+C(\va) |v|_{0;\mathcal Q_T},\\ 
\label{interpolation:2s}
|P_{\sigma}v|_{0;\mathcal Q_T}&\leq \va [v]_{2\sigma+\al,1+\frac{\al}{2\sigma};\mathcal Q_T}+C(\va)|v|_{0;\mathcal Q_T},\\ 
\label{interpolation:3s}
[v]_{\al,\frac{\al}{2\sigma};\mathcal Q_T}&\leq \va [v]_{2\sigma+\al,1+\frac{\al}{2\sigma};\mathcal Q_T}+C(\va) |v|_{0;\mathcal Q_T}.
\end{align}
\end{lem}
\begin{proof}
Using stereographic projections, \eqref{P-sigma to frac lap} and noticing that $|x-y|\geq C_n|F(x)-F(y)|$, the above inequalities follows from Lemma \ref{interpolation inequality}.
\end{proof}

\begin{lem}\label{lemma of inner product}
Suppose that $0<\al<\min(1,2\sigma)$ and $2\sigma+\al$ is not an integer. Let $u\in\mathcal C^{2\sigma+\al,1+\frac{\al}{2\sigma}}(Q_1)$ and $\eta\in C^{2}_c(\R^{n+1})$, then for any $\va>0$, there exists $C(\va)>0$ depending only on $\al,\sigma, n,\va$ and $\|\eta\|_{C^2(\R^{n+1})}$ such that
\be\label{estimate of inner product} 
[\langle u,\eta\rangle]_{\al,\frac{\al}{2\sigma};Q_1}\leq  \va [u]_{2\sigma+\al,1+\frac{\al}{2\sigma};Q_1} + C(\va)|u|_{0,Q_1} .
\ee
\end{lem}
\begin{proof}
We denote $C$ as various constants depending only on $n,\sigma,\al, \|\eta\|_{C^2(\R^{n+1})}$, and $C(\va)$ as various constants depending only on $n,\sigma,\al, \|\eta\|_{C^2(\R^{n+1})}$ and $\va$.
Recall that $\langle u,\eta\rangle$ is defined in \eqref{fractional inner product}. For any $(x,t)\in Q_1$, 
\[
\begin{split}
| \langle u,\eta\rangle(x,t)|&\le c(n,\sigma)\int_{\R^n\backslash B_{1}(x)}\frac{|u(x,t)-u(y,t)||\eta(x,t)-\eta(y,t)|}{|x-y|^{n+2\sigma}}\,\ud y\\
 &\quad +c(n,\sigma)\int_{B_1(x)}\frac{|u(x,t)-u(y,t)||\eta(x,t)-\eta(y,t)|}{|x-y|^{n+2\sigma}}\,\ud y\\
 &\leq C |u|_{0,Q_1} + C [u(\cdot,t)]_{\sigma,\R^n}\leq \va [u]_{2\sigma+\al,1+\frac{\al}{2\sigma};Q_1} + C(\va)|u|_{0,Q_1} .
\end{split}
\]
Fix any $X_1=(x_1,t_1), X_2=(x_2,t_2)\in Q_1$, $X_1\neq X_2$. For convenience, we write $\rho=\rho(X_1,X_2)$ and $u^z(x,t)=u(x,t)-u(x+z,t)$. We may also suppose that $\rho\leq 1$. 

\[
\begin{split}
&|\langle u,\eta\rangle(x_1,t_1)-\langle u,\eta\rangle(x_2,t_2)| \\
& \leq\left|\int_{|z|\leq \rho}\frac{\big(u^z(x_1,t_1)-u^z(x_2,t_2)\big)\eta^z(x_1,t_1)}{|z|^{n+2\sigma}} \,\ud z\right|\\ 
&+\left| \int_{|z|\leq \rho}\frac{\big(\eta^z(x_1,t_1)-\eta^z(x_2,t_2)\big)u^z(x_2,t_2)}{|z|^{n+2\sigma}} \,\ud z\right|\\
&+\left|\int_{|z|\geq \rho}\frac{\big(u^z(x_1,t_1)-u^z(x_2,t_2)\big)\eta^z(x_1,t_1)}{|z|^{n+2\sigma}} \,\ud z\right|\\ 
&+\left| \int_{|z|\geq \rho}\frac{\big(\eta^z(x_1,t_1)-\eta^z(x_2,t_2)\big)u^z(x_2,t_2)}{|z|^{n+2\sigma}} \,\ud z\right|\\
&:=I_1+I_2+I_3+I_4
\end{split}
\]
For $I_1$ and $I_2$, we first consider that $2\sigma+\al<1$. Then by change of variable, 
\[
\begin{split}
I_1+I_2\leq C\max_{i=1,2}[u(\cdot, t_i)]_{\al+\sigma;\R^n}\int_ {|z|\leq \rho} |z| ^{\al+\sigma+1-n-2\sigma}\ud z \leq C\max_{i=1,2}[u(\cdot, t_i)]_{\al+\sigma;\R^n}\rho^{1+\al-\sigma}.
\end{split}
\]
If $1<\al+2\sigma<2$, we have
\[
\begin{split}
I_1+I_2\leq C\max_{i=1,2}[u(\cdot, t_i)]_{\al+2\sigma-1;\R^n}\int_ {|z|\leq \rho} |z| ^{\al+2\sigma-n-2\sigma}\ud z\leq C\max_{i=1,2}[u(\cdot, t_i)]_{\al+2\sigma-1;\R^n}\rho^{\al}.
\end{split}
\]
If $2\sigma+\al>2$, then
\[
\begin{split}
I_1&\leq\left|\int_{|z|\leq \rho}\frac{\big(u^z(x_1,t_1)+\nabla_x u(x_1,t_1)z-u^z(x_2,t_2)-\nabla_x u(x_2,t_2)z\big)\eta^z(x_1,t_1)}{|z|^{n+2\sigma}} \,\ud z\right|\\ 
&+\left|\int_{|z|\leq \rho}\frac{\big(\nabla_x u(x_1,t_1)z-\nabla_x u(x_2,t_2)z\big)\eta^z(x_1,t_1)}{|z|^{n+2\sigma}} \,\ud z\right|\\ 
&\leq C\sup_{Q_1}|\nabla^2_x u|\int_{|z|\leq \rho}|z|^{3-n-2\sigma}\,\ud z + C[\nabla_x u]_{\al,\frac{\al}{2\sigma};Q_T}\rho^{\al}\int_{|z|\leq \rho}|z|^{2-n-2\sigma}\,\ud z \\
&\leq \rho^{\al}(\va [u]_{2\sigma+\al,1+\frac{\al}{2\sigma};Q_T}+C(\va) |u|_{0;Q_T}).
\end{split}
\]
Similarly,
\[
\begin{split}
I_2&\leq C|\nabla_x u|_{0;Q_1}\int_{|z|\leq \rho}|z|^{3-n-2\sigma}\,\ud z + C|\nabla_x u|_{0;Q_1}\rho^{\al}\int_{|z|\leq \rho}|z|^{2-n-2\sigma}\,\ud z \\
&\leq \rho^{\al}(\va [u]_{2\sigma+\al,1+\frac{\al}{2\sigma};Q_T}+C(\va) |u|_{0;Q_T}).
\end{split}
\]

For $I_3$ and $I_4$ we first consider that $\sigma\le\frac{1}{2}$. 
Choose an $\al'>\al$ but sufficiently close to $\al$ such that $\al'<\min(1,2\sigma)$, then
\[
\begin{split}
I_3\leq [u]_{\al',\frac{\al'}{2\sigma};Q_1}\rho^{\al'}C\int_{|z|\geq \rho} |z|^{2\sigma+\al-\al'-n-2\sigma} \,\ud z\leq C[u]_{\al',\frac{\al'}{2\sigma};Q_1}\rho^{\al},
\end{split}
\]
\[
\begin{split}
I_4 \leq C\rho^{\al'} [u(\cdot,t_2)]_{2\sigma+\al-\al';\R^n}\int_{|z|\geq \rho} |z|^{2\sigma+\al-\al'-n-2\sigma} \,\ud z\leq C[u(\cdot,t_2)]_{2\sigma+\al-\al';\R^n}\rho^{\al}.
\end{split}
\]
If $\sigma>\frac{1}{2}$ and $2\sigma+\al<2$, then
\[
\begin{split}
I_3\leq [u]_{2\sigma+\al-1,\frac{2\sigma+\al-1}{2\sigma};Q_1}\rho^{2\sigma+\al-1}C\int_{|z|\geq \rho} |z|^{1-n-2\sigma} \,\ud z \leq C[u]_{2\sigma+\al-1,\frac{2\sigma+\al-1}{2\sigma};Q_1}\rho^{\al},
\end{split}
\]
\[
\begin{split}
I_4 \leq C \rho^{2\sigma+\al-1} |\nabla u(\cdot,t_2)|_{0;\R^n}\int_{|z|\geq \rho} |z|^{1-n-2\sigma} \,\ud z \leq C|\nabla u(\cdot,t_2)|_{0;\R^n}\rho^{\al}.
\end{split}
\]
If $2\sigma+\al>2$, then for $\rho\leq |z|\leq 1$, we have
\[
|u^z(x_1,t_1)-u^z(x_2)|\leq |\nabla_x^2 u|_{0;Q_1}|x_1-x_2||z|+|u_t|_{0;Q_1}|t_1-t_2|\leq |\nabla_x^2 u|_{0;Q_1}\rho|z|+|u_t|_{0;Q_1}|\rho^{2\sigma}.
\]
Hence,
\[
\begin{split}
I_3&\leq \left|\int_{1\geq |z|\geq \rho}\frac{\big(u^z(x_1,t_1)-u^z(x_2,t_2)\big)\eta^z(x_1,t_1)}{|z|^{n+2\sigma}} \,\ud z\right|\\
\quad&+\left|\int_{|z|\geq 1}\frac{\big(u^z(x_1,t_1)-u^z(x_2,t_2)\big)\eta^z(x_1,t_1)}{|z|^{n+2\sigma}} \,\ud z\right|\\
&\leq C |\nabla_x^2 u|_{0;Q_1}\rho\int_{1\geq |z|\geq \rho}|z|^{2-n-2\sigma}\,\ud z+C|u_t|_{0;Q_1}\rho^{2\sigma}\int_{1\geq |z|\geq \rho}|z|^{1-n-2\sigma}\\
\quad&+[u]_{\al,\frac{\al}{2\sigma};Q_1}\rho^{\al}\int_{|z|\geq 1}|z|^{-n-2\sigma}\,\ud z\\
\leq &C(|\nabla_x^2 u|_{0;Q_1}+|u_t|_{0;Q_1}+[u]_{\al,\frac{\al}{2\sigma};Q_1})\rho^{\al}.
\end{split}
\]
Similarly, for $I_4$ we have
\[
I_4\leq C|\nabla_x u|_{0;Q_1}\rho^{\al}.
\]
Combining these and applying some interpolation inequalities in Lemma \ref{interpolation inequality}, we reach \eqref{estimate of inner product}.
\end{proof}

\bigskip

\noindent Tianling Jin

\noindent Department of Mathematics, Rutgers University\\
110 Frelinghuysen Road, Piscataway, NJ 08854, USA\\
Email: \textsf{kingbull@math.rutgers.edu}

\medskip

\noindent Jingang Xiong

\noindent School of Mathematical Sciences, Beijing Normal University\\
Beijing 100875, China\\[1mm]
\noindent and\\[1mm]
\noindent Department of Mathematics, Rutgers University\\
110 Frelinghuysen Road, Piscataway, NJ 08854, USA\\[0.8mm]
Email: \textsf{jxiong@mail.bnu.edu.cn/jxiong@math.rutgers.edu}

\end{document}